\documentclass[11pt]{amsart}

\usepackage{amsmath,amsthm,indentfirst}
\usepackage{amssymb}
\usepackage{amsfonts}
\usepackage{amscd}
\usepackage[latin1]{inputenc}
\usepackage{hyperref}
\usepackage{ifthen, amsfonts, amssymb, graphicx, srcltx, mathrsfs, xfrac,
amsmath}
\usepackage{enumerate}

\theoremstyle{plain}
\newtheorem*{maintheorem*}{Main Theorem}

\newtheorem*{thmd*}{Theorem 1.5}
\newtheorem*{thme*}{Theorem 1.6}
\newtheorem*{remark*}{Remark}
\newtheorem*{conjecture*}{Conjecture}
\newtheorem*{prop*}{Proposition}
\newtheorem{thm}{Theorem}[section]

\newtheorem{lem}[thm]{Lemma}
\newtheorem{prop}[thm]{Proposition}

\theoremstyle{definition}

\newtheorem*{proofc*}{Proof of Theorem C}

\newtheorem{remark}[thm]{Remark}
\newtheorem{notation}[thm]{Notation}

\DeclareMathOperator{\CAT}{CAT}
\DeclareMathOperator{\Teich}{Teich}
\DeclareMathOperator{\Mod}{Mod}

\DeclareMathOperator{\WP}{WP}
\DeclareMathOperator{\base}{base}
\DeclareMathOperator{\tw}{tw}

\DeclareMathOperator{\inj}{inj}

\DeclareMathOperator{\diam}{diam}

\DeclareMathOperator{\nonannular}{non-annular}

\DeclareMathOperator{\I}{lm}

\numberwithin{equation}{section}

\renewcommand{\bold}[1]{\medskip \noindent {\bf \boldmath #1
                        }\nopagebreak[4]}   
  
\begin{document}


\title[Recurrent WP rays with non-uniquely ergodic
Laminations]{Recurrent Weil-Petersson geodesic rays with non-uniquely
  ergodic ending laminations} 

\date{\today}

%    author one information

\author{Jeffrey Brock}
\address{Department of Mathematics, Brown University, Providence, RI, }
\email{brock@math.brown.edu}
\author{Babak Modami}
\address{ Department of Mathematics, University of Illinois at Urbana-Champaign, 1409 W Green ST, Urbana, IL}
\email{bmodami@illinois.edu}
\thanks{The first author was partially supported by NSF grant DMS-1207572.}

%    The 2010 edition of the Mathematics Subject Classification is
%    now available.  If you are citing a classification from the
%    new scheme, use the following input coding instead.
%\subjclass[2010]{Primary }

%\keywords{ }

%%%%%%%%%%%%%%%%%%%%%%%%%%%%%% Abstract %%%%%%%%%%%%%%%%%%%%%%%%%%%%%%%%%

\begin{abstract}
  We construct Weil-Petersson (WP) geodesic rays with minimal filling
  non-uniquely ergodic ending lamination which are recurrent to a
  compact subset of the moduli space of Riemann surfaces. This
  construction shows that an analogue of Masur's criterion for
  Teichm\"{u}ller geodesics does not hold for
  WP geodesics.
\end{abstract}

%%%%%%%%%%%%%%%%%%%%%%%%%%%%%%%%table of contents%%%%%%%%%%%%%%%%%%%%%%%%%%%
\maketitle

\setcounter{tocdepth}{3}
\tableofcontents

%%%%%%%%%%%%%%%%%%%%%%%%%%%%%% Introduction %%%%%%%%%%%%%%%%%%%%%%%%%%%%%
\section{Introduction}

The Weil-Petersson (WP) metric on Teichm\"uller space provides a
negatively curved, Riemannian alternative to the more familiar
Teichm\"uller metric, a Finsler metric whose global geometry is not
negatively curved in any general sense.  While negative curvature
allows one to harness a broad range of techniques from hyperbolic
geometry, difficulties in implementing these techniques arise from the
fact that the WP metric is {\em incomplete} and that its sectional
curvatures approach both $0$ and $-\infty$ asymptotically near the
completion.  Nevertheless, it is useful to draw analogies between
these metrics and instructive to determine which of these are robust
or obtainable through methods in negative curvature.

As an example, in \cite{bmm1} Brock, Masur and Minsky introduced a notion of
an {\em ending lamination} for WP geodesic rays, an analogue of the
vertical foliation associated a Teichm\"{u}ller geodesic ray. They
proved that the ending laminations parametrize the strong asymptote
class of {\em recurrent} WP geodesic rays. Recurrent rays are the rays whose
projection to the moduli space recurs to a compact set infinitely
often. Brock, Masur and Minsky \cite{bmm2} and Modami \cite{wpbehavior} initiated a
systematic study of the behavior of Weil-Petersson geodesics in terms
of their ending laminations and associated {\em subsurface projection
  coefficients}.  Certain diophantine-type conditions for subsurface projection
coefficients give strong control over the trajectories of the
corresponding geodesics.

For example, criteria on these coefficients can be given to guarantee
that geodesics projected to the moduli space stay in a compact part of
the moduli space \cite{bmm2}, recur to a compact part of the moduli space, or
diverge in the moduli space \cite{wpbehavior}. A simple scenario arises from bounding
the subsurface coefficients associated to the ending lamination of all
proper subsurfaces from above, akin to bounded-type irrational
numbers, all of whose continued fraction coefficients are bounded. In this
{\em bounded type} case the projection of the corresponding WP
geodesic to the moduli space stays in a compact subset; we say the
geodesic is {\em co-bounded}.

In this paper we prove 
\begin{thm}\label{thm : recurnue}
There are Weil-Petersson geodesic rays in the Teichm\"{u}ller space with minimal,
filling, non-uniquely ergodic ending lamination whose
projections to the moduli space are recurrent. Moreover, these rays are not contained in any compact subset of the moduli space.
\end{thm}

The theorem sits in contrast with the following result of H. Masur about
Teichm\"{u}ller geodesic rays with (minimal) non-uniquely ergodic
vertical foliation. Note that a Teichm\"{u}ller geodesic ray starting
at a point $X$ in the Teichm\"{u}ller space is determined by a unique
holomorphic quadratic differential on $X$. For the description of Teichm\"{u}ller geodesics in terms of holomorphic quadratic differentials and the associated vertical and horizontal measured foliations see e.g. \cite{rshteich}.

\begin{thm}\textnormal{(Masur's criterion)} \cite[Theorem 1.1]{masurcriterion} \label{thm
    : masurcrit} Suppose that the vertical foliation of a quadratic
  differential determining a Teichm\"{u}ller geodesic ray is 
  not uniquely ergodic. Then the Teichm\"{u}ller geodesic is
  divergent in the moduli space.
\end{thm}
\begin{remark}
 Masur states and proves the theorem with the assumption that the vertical foliation is minimal. The same argument for each minimal component of the vertical foliation gives the Theorem \ref{thm : masurcrit}.  
\end{remark}
The contrapositive of the above theorem ensures that the vertical
foliation (lamination) of a recurrent Teichm\"{u}ller geodesic is
uniquely ergodic. Comparing this fact and Theorem \ref{thm : recurnue}
exhibits an essential disparity between how the behavior of
a Teichm\"{u}ller geodesic is encoded in its vertical foliation
(lamination) and how the behavior of a Weil-Petersson geodesic is
encoded in its forward ending lamination.

\begin{remark} 
We remark that the methods here use explicit strong
control over the family of geodesics in the Weil-Petersson metric with
{\em bounded non-annular combinatorics} \cite{bmm2}. We remark that in
the low-complexity cases of the five-holed sphere and two-holed
torus, the more complete control over Weil-Petersson geodesics
obtained in \cite{bm} allows one to apply Theorem \ref{thm :
  geodlimit} to show that any Weil-Petersson geodesic with a filling
ending lamination is recurrent.  In this setting, then the mere
existence of non-uniquely ergodic filling laminations shows the
failure of Masur's Criterion in this setting.  Here, we have chosen an
explicit constructive approach that naturally generalizes to higher
genus cases.
\end{remark}

\bold{Acknowledgement:} We would like to thank Chris Leininger and
Kasra Rafi for very useful discussions. We would also like to thank
Yair Minsky for continued support and encouragement.

\section{Background}
In this paper we use the following notation:
\begin{notation}
Let $f,g:X\to \mathbb{R}^{\geq 0}$ be two function. Let $K\geq 1$ and $C\geq 0$ be two constants. We denote $f\asymp_{K,C} g$ if 
$$\frac{1}{K}g(x)-C\leq f(x)\leq Kg(x)+C$$
holds for every $x\in X$.
\end{notation}
\subsection{The Curve Complex}\label{subsec : cc}
Let $S=S_{g,n}$ be a finite type, orientable surface with genus $g$ and
$n$ punctures or boundary components. We define the complexity of the surface by
$\xi(S):=3g-3+n$. The curve complex of $S$ denoted by $\mathcal{C}(S)$
is a flag complex. When $\xi(S)>1$: Each vertex in the complex is the
isotopy class of an essential, simple closed curve. An {\it essential curve} is a curve which is not isotopic to a point, puncture or a boundary component of $S$. An edge
corresponds to a pair of isotopy classes of simple closed curves with disjoint representatives
on the surface. The curve complex is the flag complex obtained
from the first skeleton i.e we have a $k$ dimensional simplex
corresponding to any $k+1$ vertices with an edge between any pair of
them. Assigning length one to each edge makes the first skeleton of
the curve complex a metric graph. When $\xi(S)=1$, $S$ is a four-holed
sphere or a one-holed torus. The definition of the curve complex is the
same, except disjoint representative is replaced with 
intersection number $2$ or $1$, respectively.  

An {\em essential subsurface} of $S$ is a connected, closed, properly embedded subsurface
 $Y\subseteq S$, so that each boundary curve of $Y$ is either an essential simple closed curve of $S$ or a boundary curve of $S$, and $Y$ itself is not three-holed sphere.  We frequently
consider the inclusion of subcomplexes $\mathcal{C}(Y) \subseteq \mathcal{C}(S)$
induced by restriction.

For an essential annular subsurface $Y$ with core curve $\alpha$, the curve complex has a
slightly more involved definition, but a simple model: it is
quasi-isometric to $\mathbb{Z}$.  Formally, let $\langle \alpha\rangle$ be the cyclic subgroup of $\pi_{1}(S)$ generated by $\alpha$ acting on the Poincar\'{e} disk  $\mathbb{D}^{2}$ the universal cover of $S$. Let $\widetilde{Y}=\mathbb{D}^{2}/\langle\alpha\rangle$ be the annular cover of $S$ to which $Y$ lifts homeomorphically. Let $\widehat{Y}=\mathbb{D}^{2}\cup\Omega_{\alpha}/\langle\alpha\rangle$ be the natural compactification of $\widetilde{Y}$, where $\Omega_{\alpha}$ is the complement of the fixed points of $\alpha$ acting on the circle at infinity of $\mathbb{D}^{2}$. Each vertex of $\mathcal{C}(Y)$ corresponds to
the homotopy class of an arc connecting the two boundaries of
$\widehat{Y}$ relative to the boundary. There is an edge between any
two vertices corresponding to arcs with disjoint interior. We denote $\mathcal{C}(Y)$ by $\mathcal{C}(\alpha)$ as well. For more
detail see $\S 2$ of \cite{mm2}.

 We do not distinguish between the isotopy class of closed curve and any curve representing the class. A {\it multi-curve} is a collection of pairwise disjoint simple closed curves. 
 
Masur-Minsky \cite{mm1} showed that the curve complex of a surface
 $S$ is $\delta-$hyperbolic where $\delta$ depends only on the
 topological type of the surface.  Indeed, it has recently been shown
 that $\delta$ is universal, and can be taken to be the constant $17$,
 \cite{unicorns} (see alos \cite{aougabuhcc}).

\begin{notation}
We say that curves $\alpha,\beta\in \mathcal{C}_{0}(S)$ {\em overlap}
if $\alpha$ and $\beta$ cannot be realized by disjoint curves on $S$. If
$\alpha$ and $\beta$ overlap we say that $\alpha \pitchfork \beta$ holds. A
curve $\alpha$ overlaps a subsurface $Y$ if $\alpha$ can not be
realized disjoint from $Y$; we denote it by $\alpha\pitchfork Y$. Multi-curves $\sigma$ and $\sigma'$
overlap if some $\alpha\in \sigma$ and some $\alpha'\in\sigma'$
overlap. Similarly a multi-curve $\sigma$ and a subsurface $Y$
overlap if some $\alpha\in \sigma$ and $Y$ overlap. 

Let $Y$ and $Z$ be essential subsurfaces. We say that $Y$ and $Z$ {\em
  overlap} if
$\partial{Z}\pitchfork Y$ and $\partial{Y}\pitchfork Z$ hold. 
\end{notation}

\noindent{\bf Pants decompositions and markings:} A {\em pants decomposition} $P$ is a multi-curve with maximal number of curves $\xi(S)$. A {\em (partial) making} $\mu$ consists of a pants deposition $P$ and $t_\alpha$ a diameter $1$ subset of $\mathcal{C}_{0}(\alpha)$ for (some) all $\alpha\in P$. The subset of $\mathcal{C}_{0}(\alpha)$ can be represented by transversal curves to $\alpha$ on $S$. We call $P$ the base of the marking and denote it by $\base(\mu)$. 

The {\it pants graph} of $S$ denoted by $P(S)$ is a graph with vertices the pants decompositions. An edge is between any two pants decompositions that differ by an elementary move. An elementary move on a pants decomposition $P$ fixes all the curves and replaces one curve $\alpha$ with a curve in $S\backslash\{P-\alpha\}$ whose intersection number with $\alpha$ is $1$ if $S\backslash\{P-\alpha\}$ is a one-holed torus, and is $2$ if $S\backslash\{P-\alpha\}$ is a four-holed sphere. Assigning length one to each edge we obtain a metric graph.   

\bold{Laminations and foliations:} Fix a complete hyperbolic metric on
$S$. A {\it geodesic lamination} $\lambda$ is a closed subset of
$S$ consisting of disjoint, complete, simple geodesics. Each one of the geodesics is called a leaf of $\lambda$. Let $\widetilde{S}=\mathbb{D}^{2}$ be the universal cover of $S$. Denote the circle at infinity of the Poincar\'{e} disk $\mathbb{D}^{2}$ by $\widetilde{S}_{\infty}$. Let $M_{\infty}(S)$ denote $(\widetilde{S}_{\infty}\times \widetilde{S}_{\infty}\backslash \Delta)/\sim$, where $\Delta$ is the diagonal and $\sim$ is the equivalence relation generated by $(x,y)\sim(y,x)$. Since the geodesics in $\mathbb{D}^{2}$ are parametrized by points of $M_{\infty}$ the preimage of a geodesic lamination determines a closed subset of $M_{\infty}(S)$ which is invariant under the action of $\pi_{1}(S)$.
We denote the space of
geodesic laminations on $S$ equipped with the Hausdorff topology
of closed subsets of $M_{\infty}(S)$ by $\mathcal{GL}(S)$. The space $\mathcal{GL}(S)$ is a compact space. For more detail see $\S$I.4 of \cite{notesonthurston}. A {\it transverse
measure} $m$ on $\lambda$ is a measure on the set of arcs on $S$
 which is invariant under isotopies of $S$
preserving $\lambda$. The measure of an arc $a$ such that $a\subset\lambda$ or $a\cap\lambda=\emptyset$ is $0$ and otherwise the measure of $a$ is positive. A pair of a geodesic lamination $\lambda$ and a
transverse measure $m$ of $\lambda$ is a {\it measured (geodesic) lamination},
denoted by $\mathcal{L}=(\lambda,m)$. We say that $\lambda$ is the
{\em support} of the measured lamination. We denote the space of measured
laminations of $S$ equipped with the weak$^{*}$ topology by
$\mathcal{ML}(S)$. The space of projective measured laminations
$\mathcal{PML}(S)$ is the quotient of $\mathcal{ML}(S)$ with the
natural action of $\mathbb{R}^{+}$ rescaling the measures equipped with
the quotient topology. For any $\mathcal{L}\in\mathcal{ML}(S)$, let $[\mathcal{L}]$ denote the projective class of $\mathcal{L}$.

A geodesic lamination $\lambda$ is {\it minimal} if every leaf of
$\lambda$ is dense in $\lambda$. The geodesic lamination $\lambda$
{\em fills the surface} $S$ or is {\em filling} if $S\backslash \lambda$ is the
union of topological disks and annuli with core curve isotopic to a boundary curve of $S$. Equivalently, if for any simple
closed curve $\alpha$, and any transverse measure $m$ on $\lambda$, we have
$i(\alpha,(\lambda,m))>0$. Here 
$$i:\mathcal{ML}(S)\times\mathcal{ML}(S)\to \mathbb{R}^{\geq 0}$$
 denotes the natural extension of the intersection number of curves to the space of measured geodesic laminations; see \cite{bonlam}. 

Given $[\mathcal{L}] \in \mathcal{PML}(S)$, let $|\mathcal{L}|$
be the support of $\mathcal{L}$.  Then taking the quotient 
$$\mathcal{PML}(S)/|.|$$ of $\mathcal{PML}(S)$ by forgetting the measure,
the {\em ending lamination space} 
$$\mathcal{EL}(S)\subset \mathcal{PML}(S)/|.|$$ 
is the image of projective measured laminations with minimal filling support equipped with the quotient topology of the topology of $\mathcal{PML}(S)$.
\medskip

Recall that the curve complex of $S$ is a $\delta-$hyperbolic space. The following result of Klarriech describes the Gromov
boundary of the curve complex. 
\begin{prop}\cite{bdrycc}\label{prop : bdrycc}
  There is a homeomorphism $\Phi$ from the Gromov boundary of
  $\mathcal{C}(S)$ equipped with its standard topology to
  $\mathcal{EL}(S)$. Let $\{\alpha_{i}\}_{i=0}^{\infty}$ be a sequence of
  curves in $\mathcal{C}_{0}(S)$ that converges to a point $x$ in the
  Gromov boundary of $\mathcal{C}(S)$. Regarding each $\alpha_{i}$ as
  a projective measured lamination, any accumulation point of the sequence
  $\{\alpha_{i}\}_{i=0}^{\infty}$ in $\mathcal{PML}(S)$ is supported on $\Phi(x)$.
\end{prop}

A {\it singular foliation} $\mathcal{F}$ on $S$ is a foliation of the
complement of a finite set of points in $S$ called singular points. At a regular (not a singular) point
$\mathcal{F}$ is locally modeled on an open set $U\subset \mathbb{C}$ containing the origin whose leaves are the
horizontal coordinate lines. More precisely, there is a coordinate chart $x+iy$ such that the leaves of
$\mathcal{F}$ are the trajectories given by $y=\text{constant}$. At
singular points the foliation is locally modeled on an open set $U\subset \mathbb{C}$ containing the origin whose leaves are the trajectories along which the real valued $1-$form $\I (\sqrt{z^{k}dz^{2}})$ vanishes, where $k\in\mathbb{N}$. The singular point is mapped to the origin. A foliation is {\it minimal} if any half leaf of the foliation is dense in the surface.

A {\it transverse measure} on a singular foliation $\mathcal{F}$ is a measure on the
collection of arcs in the surface transversal to $\mathcal{F}$ which
is invariant under isotopies of the surface that preserve the
foliation. 

A pair consisting of a foliation and a transverse measure
on the foliation is a {\it measured foliation}. Given a foliation $\mathcal{F}$,
let $x+iy$ be a coordinate chart as above. Then $|dy|$ defines a transverse measure
on $\mathcal{F}$.

We denote the space of measured foliations of the surface $S$ equipped with the weak$^{*}$ topology by $\mathcal{MF}(S)$. For more detail see expos\'{e} $5$ of \cite{FLP}. 
\medskip

There is a one to one correspondence between measured laminations and measured
foliation up to Whitehead moves and isotopies of foliations on a surface
\cite{fol=lam}. A lamination is minimal if and only if the
corresponding foliation is minimal, see \cite[Theorem 2]{fol=lam}.
\medskip

\bold{Subsurface coefficients:} Let $Y\subseteq S$ be an essential non-annular
subsurface. Masur and Minsky \cite{mm1} define {\it subsurface projection map} 
$$\pi_{Y}:\mathcal{GL}(S)\to \mathcal{PC}_{0}(Y)$$
that assigns to $\lambda \in \mathcal{GL}(S)$ the subset of $\mathcal{C}_0(S)$ denoted by $\pi_Y(\lambda)$ as follows:
 Fix a complete hyperbolic metric on $S$ and realize $\lambda$ and $\partial{Y}$ geodesically. If $\lambda$ does not intersect $Y$, then define $\pi_{Y}(\lambda)=\emptyset$. Now suppose that $\lambda$ intersects $Y$. Let $\lambda\cap Y$ be the intersection locus of $\lambda$ and the subsurface $Y$. Consider isotopy classes of arcs in $\lambda\cap Y$ with end points on $\partial{Y}$ or at cusps of the hyperbolic surface, where the end points of arcs are allowed to move in $\partial{Y}$. For any arc $a$ (up to isotopy) take the essential boundary curves of a regular neighborhood of $a\cup\partial{Y}$ in $Y$. The union of these curves where we select one arc from each isotopy class and the closed curves in $\lambda\cap Y$ is $\pi_{Y}(\lambda)$. Note that the diameter of $\pi_{Y}(\lambda)$ viewed as a subset of $\mathcal{C}(Y)$ is at most $2$. 

Let $Y$ be an essential annular subsurface with core curve $\alpha$.  Denote the natural compactification of the annular cover of $S$ to
which $Y$ lifts homeomorphically by $\widehat{Y}$. Given a geodesic lamination $\lambda$, the projection of $\lambda$ to $Y$ is the set of component arcs of the lift of $\lambda$ to $\widehat{Y}$ which connect the two boundaries of $\widehat{Y}$. We denote the projection map by either $\pi_{Y}$ or $\pi_{\alpha}$. For more detail see \cite[\S 2]{mm2}. 

Note that since $\mathcal{C}_{0}(S)\subset \mathcal{GL}(S)$, we have in particular the subsurface projection map
$$\pi_{Y}:\mathcal{C}_{0}(S)\to \mathcal{PC}_{0}(Y).$$

Given a multi-curve $\sigma$ and an essential subsurface $Y$, the projection of $\sigma$ onto $Y$ is the union of the projections $\pi_{Y}(\alpha)$ where $\alpha\in\sigma$. For a partial marking $\mu$ if the subsurface $Y$ is not an annulus with core curve in $\base(\mu)$, then $\pi_{Y}(\mu)=\pi_{Y}(\base(\mu))$. If $Y$ is an annulus with core curve $\alpha\in\base(\mu)$, then $\pi_{Y}(\mu)$ is the set of transversal curves to $\alpha$ in $\mu$. 
\medskip

Let $\mu$ and $\mu'$ be either partial markings or laminations. Let $Y\subseteq S$ be an essential subsurface. The $Y$
\textit{subsurface coefficient} of $\mu$ and $\mu'$ is defined by 
$$d_{Y}(\mu,\mu'):=\min\{d_{Y}(\gamma,\gamma') : \gamma\in\pi_{Y}(\mu), \gamma'\in\pi_{Y}(\mu')\}.$$ 
When $Y$ is an annular subsurface with core curve $\alpha$ we denote $d_{Y}(.,.)$ by $d_{\alpha}(.,.)$ as well.

Lemma 2.1 of \cite{mm1} gives us the following bound on the subsurface coefficient of two curves in terms of their intersection number,
\begin{equation}\label{eq : dYi}d_{Y}(\alpha,\beta)\leq 2i(\alpha,\beta)+1.\end{equation}

Let $\alpha,\beta\in\mathcal{C}_{0}(S)$ and $Y\subseteq S$ be an essential subsurface. If $d_{Y}(\alpha,\beta)>2$, then $\alpha\pitchfork\beta$ holds. To see this, first suppose that $Y$ is non-annular. Recall the surgery map which assigns to any arc in $Y$ with end points on $\partial{Y}$ the set of curves in the boundary of a regular neighborhood of $a\cup\partial{Y}$. This map from the arc complex of $Y$ to $\mathcal{PC}_{0}(Y)$ is $2-$Lipschitz, see \cite{mm2}. Let $a$ be an arc in $\alpha\cap Y$ and $b$ be an arc in $\beta\cap Y$ with end points in the boundary of $Y$. The assumption $d_{Y}(\alpha,\beta)>2$ and the fact that the surgery map is $2-$Lipschitz imply that $a$ and $b$ have arc complex distance at least $2$. Thus the arcs $a$ and $b$ intersect, and therefore the curves $\alpha$ and $\beta$ intersect. Now suppose that $Y$ is an annular subsurface. Then $d_{Y}(\alpha,\beta)>2$ implies that the interior of any two lifts of $\alpha$ and $\beta$ to the compactified annular cover $\widehat{Y}$ that go between two boundary components of $\widehat{Y}$ intersect. Therefore, $\alpha$ and $\beta$ intersect. 

The following lemma is a consequence of \cite[Lemma 2.3]{mm2}.

\begin{lem}\label{lem : diamproj}
Let $\mu$ denote a multi-curve, (partial) marking or lamination on a surface $S$. Then for any essential subsurface $Y\subseteq S$ we have
$$ \diam_{Y}(\mu)\leq 2.$$
When $Y$ is an annulus the sharp upper bound is $1$. 
\end{lem}
The reason for the second part of the lemma is that any two lifts of two disjoint curves on the surface $S$ to the compactified annular cover corresponding to the annular subsurface $Y\subset S$ are disjoint.
\medskip

Let $\alpha,\beta,\gamma\in \mathcal{C}_{0}(S)$. Farb, Lubotzky and Minsky \cite{rk1mcg} defined the {\em relative twist} of the curves $\beta$ and $\gamma$ with respect to the curve $\alpha$ by
$$\tau_{\alpha}(\beta,\gamma):=\{ b.c : b\in \pi_{\alpha}(\beta), c\in \pi_{\alpha}(\gamma)\}$$ 
 where $b.c$ denotes the algebraic intersection number of the arcs $a$ and $b$. The arcs $b$ and $c$ are oriented so that they intersect the lift of $\alpha$ homotope to the core of $\widehat{Y}$ in the same direction. More precisely, let $\tilde{\alpha}$ be the lift of $\alpha$ homotopic to the core of $\widehat{Y}$ and fix an orientation for $\tilde{\alpha}$. Then $b$ and $c$ are oriented so that the tangents to $\tilde{\alpha}$ and $b$, and the tangents to $\tilde{\alpha}$ and $c$ at their intersection points determine the same orientation for the annulus $\widetilde{Y}$. Note that The subset $\tau_{\alpha}(\beta,\gamma)\subset\mathbb{Z}$ has diameter $2$.

Given arcs $b,c\in\mathcal{C}(\alpha)$, by the discussion in $\S 2.4$ of \cite{mm2}, 
$$d_{\alpha}(b,c)=|b.c|+1.$$
 Let $\beta,\gamma\in\mathcal{C}_{0}(S)$. Since the diameter of $\tau_{\alpha}(\beta,\gamma)$ is at most $2$, by the above formula we have 
\begin{equation}\label{eq : dYiYann}\big|d_{\alpha}(\beta,\gamma)-|x|\big|\leq 3.\end{equation}
 for any $x\in\tau_{\alpha}(\beta,\gamma)$. Let $\gamma=\mathcal{D}_{\alpha}^{e}(\beta)$, where $\mathcal{D}_{\alpha}$ is the positive Dehn twist about $\alpha$ and $e$ is a positive integer. Formula (2) in \cite[\S2]{rk1mcg} is 
\begin{equation}\label{eq : rtdt}\tau_{\alpha}(\beta,\gamma)\subset\{e,e+1\}.\end{equation}

The following inequality proved by Behrstock \cite{beh} relates the subsurface
coefficients of two subsurfaces that overlap. 
\begin{thm}\textnormal{(Behrstock inequality)} \label{thm : beh} 
There is a constant $B_{0}>0$ so that given a curve system $\mu$ and
subsurfaces $Y$ and $Z$ satisfying $Y\pitchfork Z$ we have 
$$\min\{d_{Y}(\partial{Z},\mu),d_{Z}(\partial{Y},\mu)\}\leq B_{0}.$$
\end{thm}
\begin{remark} We note that Chris Leininger has observed that $B_0$
can be taken to be the universal constant $3$. However the specific value of $B_{0}$ does not play any role in our work.
\end{remark}

\bold{Limits of laminations:} Let
$\mathcal{L}_{i}=(\lambda_{i},m_{i}) (i\in\mathbb{N})$ be a sequence of measured
laminations which converges to a measured lamination
$\mathcal{L}=(\lambda,m)$ in the weak$^{*}$ topology. Suppose that
after possibly passing to a subsequence, the laminations $\lambda_{i}$ converge to a
lamination $\xi$ in the Hausdorff topology of $M_{\infty}(S)$. It is a standard fact that $\lambda\subseteq \xi$, see for
example $\S$I.4 of \cite{notesonthurston}.  
\begin{lem}\label{lem : subsurfcomplim}
 Suppose that a sequence of curves $\{\alpha_{i}\}_{i=0}^{\infty}$ converges
to a lamination $\lambda$ in the Hausdorff topology of $M_{\infty}(S)$. Let $Y$ be an essential subsurface so that $\lambda$ intersects $Y$ essentially. Then for any geodesic lamination $\lambda'$ that intersects $Y$ essentially, we have 
$$d_{Y}(\alpha_{i},\lambda')\asymp_{1,4}d_{Y}(\lambda,\lambda'),$$   
 for all $i$ sufficiently large.   
 \end{lem}
 \begin{proof}
 First suppose that $Y$ is an essential non-annular subsurface. Equip $S$ with a complete hyperbolic metric and realize $\partial{Y}$, the curves $\alpha_{i}$
 and the lamination $\lambda$ geodesically in this
 metric. Let $b$ an arc in $\lambda\cap Y$ and $\delta>0$ be so that the $\delta-$neighborhood of $b\cup\partial{Y}$ is a regular neighborhood and at least one of the components of the boundary of the neighborhood is an essential curve in $Y$. Denote the neighborhood by $U$, see Figure \ref{fig : Hdistsubsurf}. Let $l$ be the geodesic in $\lambda$ so that $b\subset l$. Let $\tilde{l}$ be a lift of $l$ to the universal cover $\mathbb{D}^{2}$.
 The convergence of the curves $\alpha_{i}$ to $\lambda$ in the Hausdorff topology of $M_{\infty}(S)$ (see Lemma I.4.1.11 in \cite[\S I.4]{notesonthurston}) guarantees that given $\delta'<\delta$ and $L>0$, for all $i$ sufficiently large,  there is a lift $\tilde{\alpha}_{i}$ of $\alpha_{i}$ to $\mathbb{D}^{2}$ so that $\tilde{\alpha}_{i}$ and $\tilde{l}$, $\delta'-$fellow travel on an interval of length at least $L$. Then projecting $\tilde{\alpha}_i$ and $\tilde{l}$ to $S$ we can see that there is an arc $a_{i}$ of $\alpha_{i}\cap Y$ such that the arcs $b$ and $a_{i}$ are $\delta'$ fellow travelers in $Y$. This implies that the regular neighborhood $U$ is also a regular neighborhood of $a_{i}\cup\partial{Y}$. By the definition of the subsurface projection the essential boundary curve of this neighborhood is a curve in $\pi_{Y}(\alpha_{i})$. 
 
   \begin{figure}
\centering
\includegraphics[scale=.16]{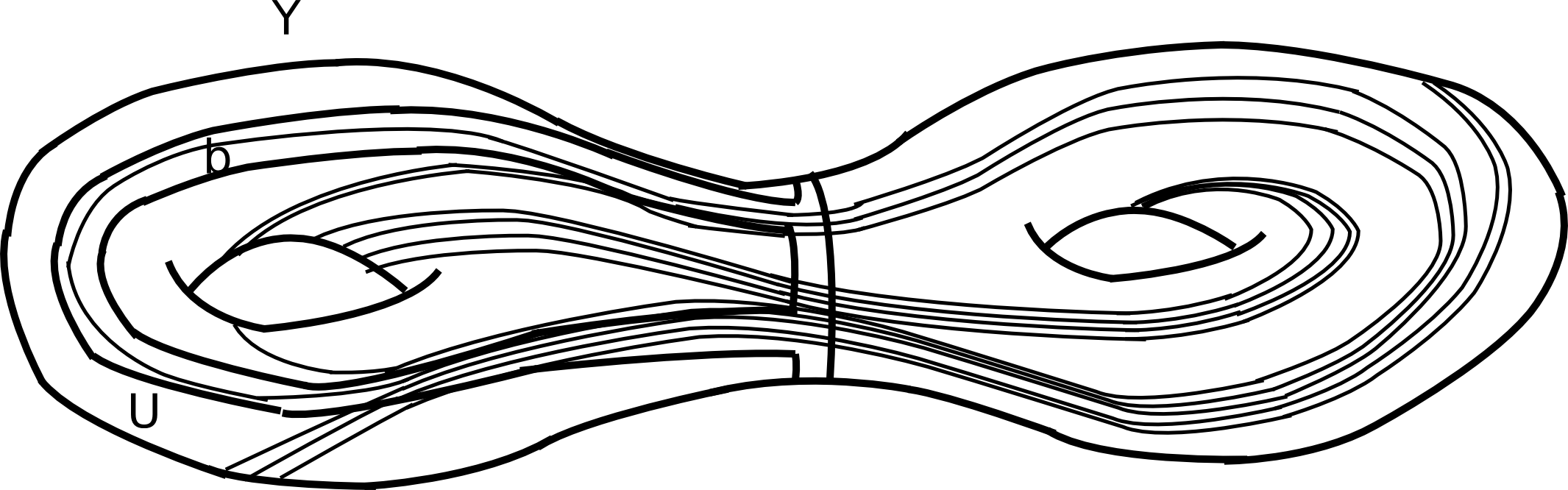}
\caption{ The $\delta$ neighborhood $U$ of an arc $b$ in $\lambda\cap Y$ is a regular neighborhood in $Y$ with at least one essential boundary curve. For $i$ sufficiently large, the Hausdorff distance of $\alpha_{i}$ and $\lambda$ is less than $\delta$, so $\alpha_{i}\cap Y$ is contained in $U$.} 
\label{fig : Hdistsubsurf}
\end{figure}

By Lemma \ref{lem : diamproj}, $\pi_{Y}(\lambda)$ and $\pi_{Y}(\alpha_{i})$ are subsets of $\mathcal{C}_{0}(Y)$ with diameter at most $2$. Moreover as we saw in the previous paragraph $\pi_{Y}(\lambda)\cap\pi_{Y}(\alpha_{i})\neq \emptyset$. Therefore 
$$\diam_{Y}(\pi_{Y}(\alpha_{i})\cup\pi_{Y}(\lambda))\leq 4.$$
Let $\beta$ be a curve in $\pi_{Y}(\lambda')$. Then by the above bound on the diameter we have
  $$|d_{Y}(\beta,\alpha_{i})-d_{Y}(\beta,\lambda)|\leq 4.$$
   This completes the proof of the lemma for non-annular subsurface $Y$.
  \medskip
 
 Now suppose that $Y$ is an essential annular subsurface with core curve $\gamma$. Let $b$ be an arc in $\pi_{Y}(\lambda)$. We claim that, for all $i$ sufficiently large, there is an arc $a_{i}$ in $\pi_{Y}(\alpha_{i})$ such that $a_{i}$ and $b$ have at most one intersection point in their interior. If $a_{i}$ and $b$ do not intersect we are done. Otherwise, after conjugation we may assume that the origin of $\mathbb{D}^{2}$ is a lift of an intersection point of $a_{i}$ and $b$. Moreover there are $\tilde{b}$ a lift of $b$ and $\tilde{a}_{i}$ a lift of $a_{i}$ to $\overline{\mathbb{D}^{2}}$ which pass through the origin, see Figure \ref{fig : annprojbd}. As in Figure \ref{fig : annprojbd}, there is a lower bound for the distance of $\tilde{b}$ and any other lift of $b$ to $\mathbb{D}^{2}$. Then choosing $\delta>0$ sufficiently small and $L>0$ large enough, any geodesic in $\mathbb{D}^{2}$ passing through the origin which $\delta$ fellow travels $\tilde{b}$ on an interval of length at least $L$ does not intersect any of the lifts of $b$ except $\tilde{b}$. The geodesic $\tilde{b}$ is a lift of a leaf of $\lambda$ to $\mathbb{D}^{2}$ and $\tilde{a}_{i}$ is a lift of $\alpha_{i}$ to $\mathbb{D}^{2}$. So the Hausdorff convergence of the curves $\alpha_{i}$ to $\lambda$ implies that given $\delta,L>0$, for $i$ sufficiently large $\tilde{a}_{i}$, $\delta$ fellow travels $\tilde{b}$ on an interval of length at least $L$. Therefore as we saw above $\tilde{a}_{i}$ intersects $\tilde{b}$ once at the origin and does not intersect any other lift of $b$. The number of times that the arcs $a_i$ and $b$ intersect is equal to the number times that $\tilde{b}$ intersects all of the lifts of $a_{i}$ to $\mathbb{D}^{2}$. Which by the above discussion is at most $1$ (see Figure \ref{fig : annprojbd}). The proof of the claim is complete.
 
 \begin{figure}
\centering
\includegraphics[scale=.16]{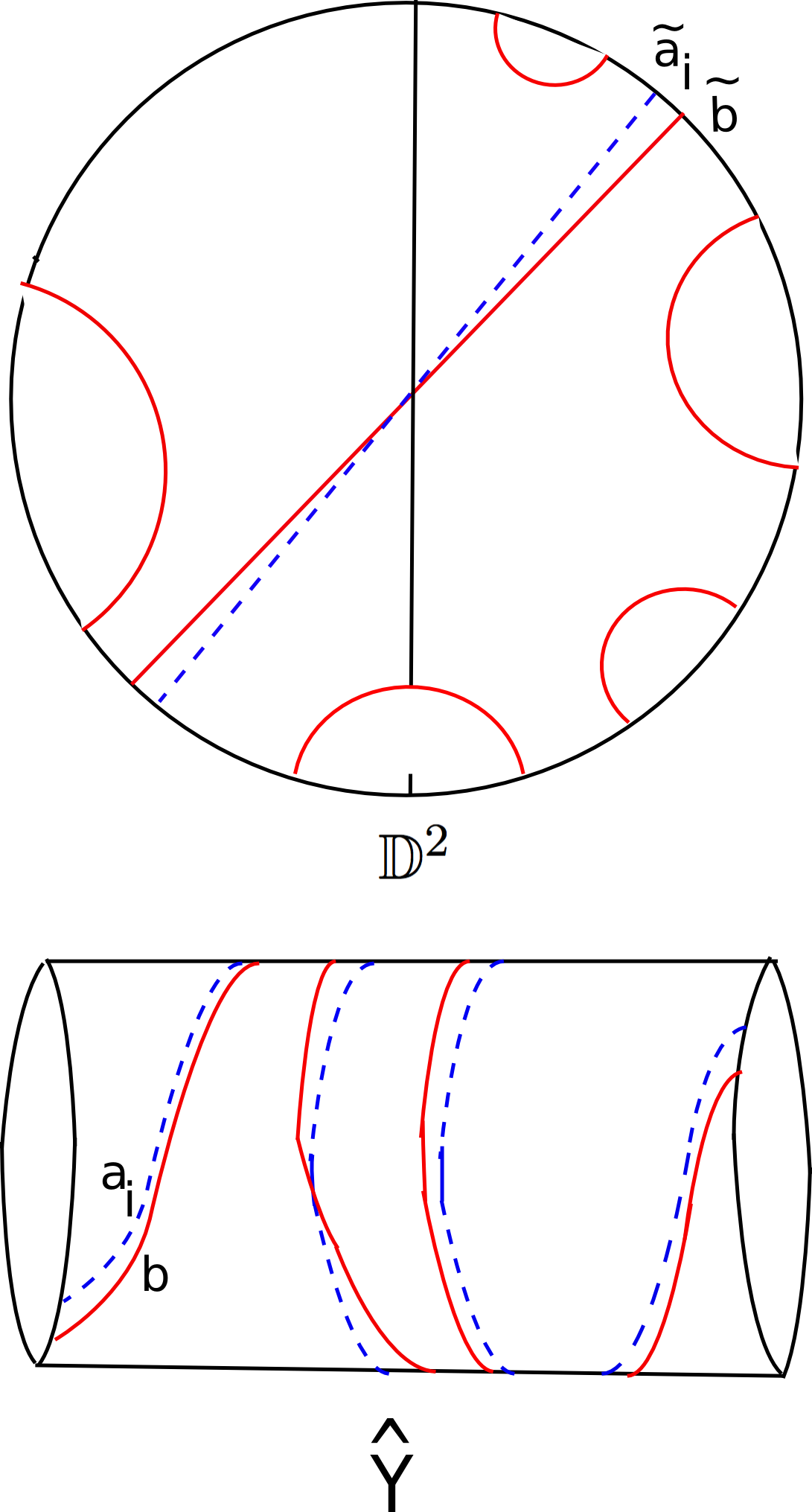}
\caption{Bottom: The arc $b$ in $\pi_{Y}(\lambda)$ and $a_{i}$ in $\pi_{Y}(\alpha_{i})$ in the compactified annular cover $\widehat{Y}$. Top: The lift $\tilde{b}$ of the arc $b$ and $\tilde{a}_{i}$ of $a_{i}$ to the universal cover $\mathbb{D}^{2}$ that pass through the origin and fellow travel for a long portion. As in the picture $\tilde{a}_{i}$ and $\tilde{b}$ intersect once, moreover $\tilde{a}_{i}$ does not intersect any other lift of $b$ to $\mathbb{D}^{2}$. Thus $b$ and $a_{i}$ in $\widehat{Y}$ intersect once.} 
\label{fig : annprojbd}
\end{figure}

 The fact that $a_{i}$ and $b$ intersect at most once implies that 
 $$d_{Y}(\alpha_{i},\lambda)\leq 2.$$
   By Lemma \ref{lem : diamproj}, $\pi_{Y}(\alpha_{i})$ and $\pi_{Y}(\lambda)$ are subsets of $\mathcal{C}_{0}(Y)$ with diameter at most $1$. Moreover as we saw above $\pi_{Y}(\alpha_i)$ and $\pi_{Y}(\lambda)$ have distance at most $2$.  Therefore 
$$\diam_{Y}(\pi_{Y}(\alpha_{i})\cup\pi_{Y}(\lambda))\leq 4.$$
Let $\beta$ be a curve in $\pi_{Y}(\lambda')$. Then by the above bound on the diameter we have
  $$|d_{Y}(\beta,\alpha_{i})-d_{Y}(\beta,\lambda)|\leq 4.$$
   This completes the proof of the lemma for annular subsurface $Y$.
  \end{proof}   
   
 \bold{Hierarchy paths and the distance formula:} Hierarchy paths
 introduced by Masur and Minsky \cite{mm2}, comprise quasi-geodesics in
 the pants and marking graphs of a surface with constants depending
 only on the topological type of the surface. Hierarchy paths have
 properties encoded in their end points and the associated subsurface
 coefficients. For a list of these properties see \cite[\S 2]{bmm2} and
 \cite[\S 2]{wpbehavior}. Here we only state a key feature of hierarchy paths which is the no backtracking property. For other properties we provide a reference wherever we use them.

 \begin{thm}\label{thm : nbacktr}
There exists a constant $M_{2}>0$ depending only on the topological type of the surface $S$ with the following property. Let $\rho:[m,n]\to P(S)$ be a hierarchy path. Let $i,j,k,l\in[m,n]$ with $i\leq j\leq k\leq l$. For any subsurface $Y\subseteq S$ we have
 $$d_{Y}(\rho(i),\rho(l))+2M_{2}\geq d_{Y}(\rho(j),\rho(k)).$$
 \end{thm}
  
 The following theorem is a
 straightforward consequence of the Bounded Geodesic Image Theorem
  \cite[Theorem 3.1]{mm2}.

\begin{thm}\label{thm : bddgeod}
Given $k\geq 1$ and $c\geq 0$, there is a $G\geq 0$ with the following
property. Let $\{\gamma_{i}\}_{i=0}^{\infty}$ be a sequence of curves in
$\mathcal{C}_{0}(S)$ which form a $1-$Lipschitz,
$(k,c)-$quasi-geodesic. Let $Y\subsetneq S$ be an essential subsurface so that
$\gamma_{i}\pitchfork Y$ holds for all $i\geq 0$, then  
$$\diam_{Y}(\{\pi_{Y}(\gamma_{i})\}_{i=0}^{\infty})\leq G.$$
\end{thm}
\noindent Here $\diam_{Y}(.)$ is the diameter of the given subset of
$\mathcal{C}(Y)$.

Using the hierarchical machinery Masur and Minsky provide the following
quasi-distance formula in the pants graph of a surface \cite[Theorem 6.12]{mm2}.  Given $A>M_{1}$ ($M_{1}$ is a
constant depending on the topological type of $S$) there are constants $K\geq 1$ and $C\geq 0$ such that 
\begin{equation}\label{eq : dsf}d(\mu,\mu')\asymp_{K,C} \sum_{\substack{Y\subseteq S\\ \nonannular}}\{d_{Y}(\mu,\mu')\}_{A}.\end{equation}
Here the cut-off function $\{.\}_{A}:\mathbb{R}\to \mathbb{R}^{\geq
  0}$ is defined by 
$$\{a\}_{A} = \Big\lbrace  
   \begin{array}{cl}
  a & \;\text{if}\; a\geq A,\\
  0 & \;\text{if}\: a < A .\\
  \end{array}$$

\bold{Bounded combinatorics:} A pair of laminations or partial
markings $(\mu,\mu')$ has non-annular $R-$bounded combinatorics if  
$$d_{Y}(\mu,\mu')\leq R$$
for every proper, essential, non-annular subsurface $Y\subsetneq S$. 

The following result about stability of hierarchy paths with
non-annular bounded combinatorics in the pants graph is an
important ingredient in the proof that bounded combinatorics of end
invariants of a WP geodesic guarantees co-boundedness of the geodesic 
and vice versa (see \cite{bmm2}). We need this theorem in our study of
the behavior of WP geodesics in $\S$\ref{sec : recur}.  

\begin{thm}\cite{bmm2}\label{thm : bddcombstable}
Given $R>0$ there is a quantifier function $$d_{R}:\mathbb{R}^{\geq
  1}\times \mathbb{R}^{\geq 0}\to \mathbb{R}^{\geq 0}$$ so that a
hierarchy path $\rho$ with end points with non-annular $R-$bounded
combinatorics is $d_{R}-$stable in the pants graph. That is, any
$(K,C)-$quasi-geodesic with end points on $|\rho|$ stays in the
$d_{R}(K,C)$ neighborhood of $|\rho|$. Here $|\rho|$ is the union of the pants decompositions of $\rho$. 
\end{thm}

\subsection{The Weil-Petersson metric}
In this section we assemble properties of the
Weil-Petersson metric we will need. For an introduction to the
synthetic geometry of the Weil-Petersson metric see \cite{wol}. 

The Weil-Petersson metric on the Teichm\"{u}ller space $\Teich(S)$ is
a Riemannian metric with negative sectional curvatures. It is
incomplete, but is geodesically convex: any two points are joined by a
unique geodesic that lies in the interior.  Its metric
completion $\overline{\Teich(S)}$ is a $\CAT(0)$ space. See  $\S$
II.3.4 of \cite{bhnpc} for an introduction to $\CAT(0)$ space.  By the work of H. Masur \cite{maswp} the
completion of the Teichm\"{u}ller space with the Weil-Petersson metric
is naturally identified with the {\em augmented Teichm\"{u}ller space}
obtained by adjoining nodal surfaces as limits. The completion is
stratified by the data of simple closed curves on $S$ that are
pinched: each stratum $\mathcal{S}(\sigma)$ is a copy of the
Teichm\"{u}ller space of the surface $S\backslash \sigma$, where
$\sigma$ is a multicurve. Masur also gave an expansion of the metric near
the completion showing that the inclusion
$\mathcal{S}(\sigma)\hookrightarrow \Teich(S)$ is an isometry and
$\mathcal{S}(\sigma)$ is totally geodesic.  

S. Yamada observed that a stronger form of Masur's expansion should
hold near the completion guaranteeing that the Weil-Petersson metric
is asymptotic to a metric product of strata to higher order, and work
of Daskalopolous-Wentworth \cite{dwwp} gave the appropriate metric
expansion.  Their expansion showed that these completion strata have
the {\em non-refraction property}: For any $X,Y\in\overline{\Teich(S)}$,  the interior of the unique geodesic connecting $X$ and $Y$ lies in the smallest stratum that contains $X$ and $Y$. See \cite{wols} for stronger form of the asymptotic expansion of the WP metric.
The Weil-Petersson metric is invariant under the action of the mapping
class group of the surface $\Mod(S)$ and descends under the natural
orbifold cover to a metric on the moduli space of Riemann surfaces
$\mathcal{M}(S)$.  The completion descends to a metric on the familiar
Deligne-Mumford compactification of $\mathcal{M}(S)$.

\bold{Length-functions:} Let $X\in \Teich(S)$. Let $\alpha$ be a
closed curve on $S$. We denote by $\ell_{\alpha}(X)$ the length of the
geodesic representative of $\alpha$ in its free homotopy class on
$S$. The length-function has a  natural extension to the space of
measured laminations, \cite{bonlam}. For $\mathcal{L}\in
\mathcal{ML}(S)$, we denote the length of $\mathcal{L}$ by $\ell_{\mathcal{L}}(X)$. 

Significant from the point of view of the Weil-Petersson geometry is
the result of Wolpert, \cite{wol}, that each length-function is a
strictly convex function along any WP geodesic. 

\bold{Quasi-isometric model:} Let $S$ be a surface with negative Euler characteristic. There is a constant $L_{S}$ (Bers
constant) depending only on the topological type of $S$ such that any
complete hyperbolic metric on $S$ has a pants decomposition (Bers pants
decomposition) with the property that the length of any curve in the pants decomposition is at most $L_{S}$, see \cite[\S 5]{buser}. Let $X\in \overline{\Teich(S)}$. Suppose that $X\in\mathcal{S}(\sigma)$. A {\it Bers pants decomposition} of $X$, denoted by $Q(X)$, is the union of Bers pants decompositions of the connected components of $S\backslash\sigma$ and $\sigma$. A {\it Bers marking} of $X$, denoted by $\mu(X)$, is a partial marking
obtained from a Bers pants decomposition $Q$ of $X$ by adding a transversal
curve with minimal length for each $\alpha\in Q-\sigma$.  The following
result of Jeffrey Brock provides a quasi-isometric model for the
Weil-Petersson metric. 

\begin{thm}\textnormal{(Quasi-isometric model)} \label{thm : brockqisom} \cite{br}
There are constants $K_{\WP}\geq 1$ and $C_{\WP}\geq 0$ depending only on
the topological type of $S$ with the following property. The map
$Q:\Teich(S)\to P(S)$, assigning to each $X$ a Bers pants
decomposition of $X$ is a $(K_{\WP},C_{\WP})-$quasi-isometry. 
\end{thm}

\bold{Ending laminations:}  Let $r:[0,a)\to \Teich(S)$ be a WP
geodesic ray.  Any limit in the weak$^{*}$ topology of an infinite sequence of
distinct Bers curves at $r(t_{n})$ where $t_{n}\to a$ is an {\it ending measured lamination} of $r$. A pinching
curve $\alpha$ along $r$ is a curve so that $\ell_{\alpha}(r(t))\to 0$
as $t\to a$. In \cite{bmm1} Brock, Masur and Minksy showed that the union of
the supports of ending measured laminations and pinching curves of $r$ is a geodesic lamination. We call this lamination the {\it ending lamination} of $r$.

Let $g:(a,b)\to \Teich(S)$ be a WP geodesic, where $(a,b)$ is an open interval containing $0$. If  the forward
trajectory $g|_{[0,b)}$ can be extended to $b$ so that $g(b)\in
\overline{\Teich(S)}$ we define the forward end invariant of $g$ to be
a Bers partial marking of $g(b)$. If not let the forward end invariant
of $g$ be the lamination of $g|_{[0,b)}$ we defined above. We denote
the forward end invariant by $\nu^{+}=\nu^{+}(g)$. Similarly, consider
the backward trajectory $g|_{(a,0]}$ and define the backward end
invariant of $g$, $\nu^{-}=\nu^{-}(g)$. 

From $\S 8$ of \cite{wpbehavior} we have the following result: 

\begin{lem}\label{lem : rayinfty} \textnormal{(Infinite rays)}
Let $\nu$ be a minimal filling lamination. There is an infinite WP
geodesic ray $r$ with forward ending lamination $\nu$.   
\end{lem}

The following strengthened version
of Wolpert's Geodesic Limit Theorem (see \cite{wols} and \cite{bmm2})
proved in $\S 4$ of \cite{wpbehavior} provides a limiting picture for a sequence of bounded length WP 
geodesic segments in the Teichm\"{u}ller space.  
\begin{thm} \textnormal{(Geodesic limits)} \label{thm : geodlimit} 
Given $T>0$. Let $\zeta_{n}:[0,T]\to \overline{\Teich(S)}$ be a sequence of WP
geodesic segments parametrized by arc-length.  
After possibly passing to a subsequence there is a partition
$0=t_{0}<....<t_{k+1}=T$ of $[0,T]$, possibly empty multi-curves
$\sigma_{0},...,\sigma_{k+1}$ and a multi-curve $\hat{\tau}\equiv
\sigma_{i}\cap \sigma_{i+1}$ for $i=0,1,...,k$ and a piece-wise
geodesic 
 $$\hat{\zeta}:[0,T]\to \overline{\Teich(S)},$$
 with the following properties
\begin{enumerate}[(1)]
\item\label{gl : tau} $\hat{\zeta}((t_{i},t_{i+1}))\subset \mathcal{S}(\hat{\tau})$ for $i=0,...,k$, 
\item \label{gl : sigma}$\hat{\zeta}(t_{i})\in \mathcal{S}(\sigma_{i})$ for $i=0,...,k+1$,

Given a multi-curve $\sigma$ denote by $\tw(\sigma)$ the subgroup of $\Mod(S)$ generated by positive Dehn twists about the curves in $\sigma$. There are elements of the mapping class group $\psi_{n}$ for each $n\in\mathbb{N}$, and $\mathcal{T}_{i,n}\in \tw(\sigma_{i}-\hat{\tau})$ for $ i=1,...,k$ and $n\in\mathbb{N}$ so that
\item\label{gl : converge}$\psi_{n}(\zeta_{n}(t))\to \hat{\zeta}(t)$ as $n\to \infty$ for all $t\in [0,t_{1}]$. Let $\varphi_{i,n}=\mathcal{T}_{i,n}\circ ...\circ \mathcal{T}_{1,n}\circ \psi_{n}$ for $i=1,...,k$ and each $n\in\mathbb{N}$. Then 
 $\varphi_{i,n}(\zeta_{n}(t))\to \hat{\zeta}(t)$ for any $t\in[t_{i},t_{i+1}]$ as $n\to \infty$.
 \end{enumerate}
\end{thm}
\begin{remark}
The central difference between the above version and original versions
lies in the assertion that we have one (possibly
empty) multi-curve $\hat{\tau}$ rather than several multi-curves
$\tau_{i}=\sigma_{i}\cap\sigma_{i+1}, i=0,1,..,k,$ allowed in Wolpert's Geodesic
Limit Theorem.  In particular, in part (\ref{gl : tau}) the geodesic segments
$\hat{\zeta}((t_{i},t_{i+1}))$ lie in one stratum
$\mathcal{S}(\hat{\tau})$ rather than several strata
$\mathcal{S}(\tau_{i})$. 
\end{remark}
 
\section{Minimal non-uniquely ergodic laminations}\label{sec : nue}

A (measurable) geodesic lamination $\lambda$ is {\it non-uniquely
  ergodic} if there are non-proportional measures supported on
$\lambda$. More precisely, $\lambda$ is non-uniquely ergodic if there exist
transverse measures $m$ and $m'$ supported on $\lambda$ and curves
$\alpha$ and $\beta$ such that
$$\frac{m(\alpha)}{m'(\alpha)}\neq\frac{m(\beta)}{m'(\beta)}.$$

 Gabai \cite [\S9]{gabaiendlamspace} gave a recipe to construct
minimal filling non-uniquely ergodic geodesic laminations on any
surface $S$ with $\xi(S)> 1$. In fact, Gabai outlined the
construction of minimal filling laminations and measures supported on each one of them with distinct projective classes
 \cite[Theorem 9.1]{gabaiendlamspace}. Leininger-Lenzhen-Rafi \cite[\S3-5]{nonuniqueerg} gave a detailed
construction of minimal filling non-uniquely ergodic
laminations on the surface $S_{0,5}$. Moreover, they studied the set of
measures supported on the lamination and their projective
classes. 

We first recall the construction of \cite{nonuniqueerg}. Let
$\{e_{i}\}_{i=1}^{\infty}$ be a sequence of positive integers. Let
$\rho:S_{0,5}\to S_{0,5}$ be the order-five homeomorphism of $S_{0,5}$
realized as the rotation by angle $\frac{4\pi}{5}$ in Figure \ref{fig : nue05}. Let $\mathcal{D}=\mathcal{D}_{\hat{\gamma}_{2}}$ be the
positive Dehn twist about the curve $\hat{\gamma}_{2}$. Let
$f_{i}=\mathcal{D}^{e_{i}}\circ \rho$, for $i\geq 1$. Define the sequence of curves
$\hat{\gamma}_{i}=f_{1}\circ f_{2}\circ ...\circ
f_{i}(\hat{\gamma}_{0})$, for $i\geq1$. The curves
$\hat{\gamma}_{0},...,\hat{\gamma}_{5}$ are shown in Figure \ref{fig : nue05}.
  
   \begin{figure}[htbp]
\centering
\includegraphics[scale=.16]{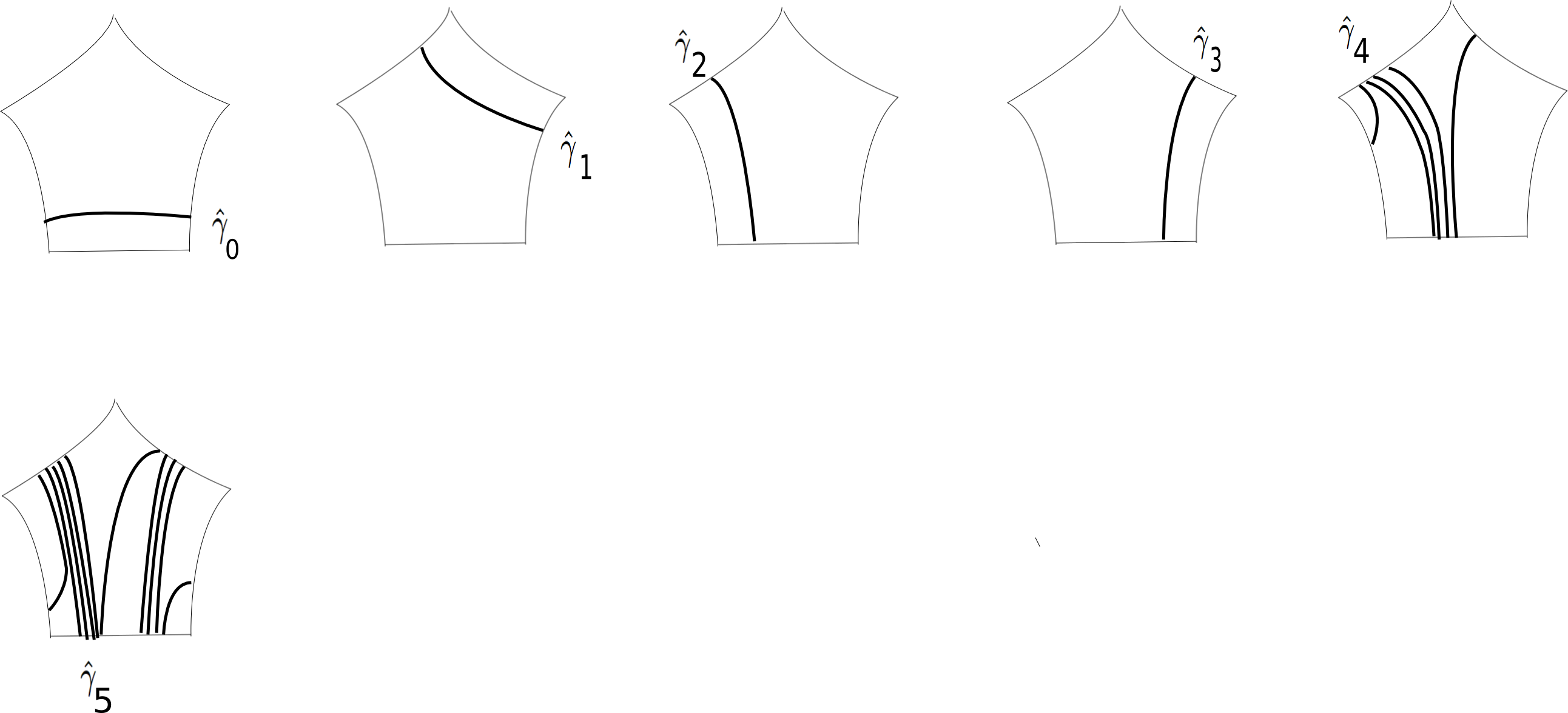}
\caption{The double of each pentagon in the picture is a five-times punctured $2-$sphere.
  Let the curves $\hat{\gamma}_{0},\hat{\gamma}_{1},...,\hat{\gamma}_{5}$ are shown in the
  picture. Any other six consecutive curves in the sequence after applying an appropriate element of $\Mod(S_{0,5})$ are the same as the above six curves, where the last two curves have different number of parallel strands from $\hat{\gamma}_4$ and $\hat{\gamma}_5$, respectively.} 
\label{fig : nue05}
\end{figure}

\begin{prop}\label{prop : 1-nue05} There are constants $E>0$, $k\geq 1, c\geq 0$, and $K\geq 1, C\geq 0$ with the following properties. Suppose that $\{e_{i}\}_{i=1}^{\infty}$ is a sequence of integers satisfying $e_{i}>E$ for all $i\geq 1$. Let $\{\hat{\gamma}_{i}\}_{i=0}^{\infty}$ be the sequence of curves described. Then:
\begin{enumerate} 
\item\label{nue05 : int} For any $i\geq 0$ and $j\geq i+2$, $\hat{\gamma}_{j}\pitchfork \hat{\gamma}_{i}$ holds.  
\item\label{nue05 : fill} For any $i\geq 0$ and $j\geq i+4$ the curves $\hat{\gamma}_{i}$ and $\hat{\gamma}_{j}$ fill the surface $S_{0,5}$.
\item \label{nue05 : qg} The sequence of curves $\{\hat{\gamma}_{i}\}_{i=0}^{\infty}$ is a $1-$Lipschitz, $(k,c)-$quasi-geodesic in $\mathcal{C}_{0}(S_{0,5})$. 
\item \label{nue05 : dgi} $d_{\hat{\gamma}_{i}}(\hat{\gamma}_{j},\hat{\gamma}_{j'})\asymp_{K,C} e_{i-1}$ for any $j\geq i+2$ and $j'\leq i-2$. 
\end{enumerate}
\end{prop}
\begin{proof}
We start by proving the following lemma.
\begin{lem}\label{lem : int-subsurfbd}
Let $i\geq 1$. For any $j\geq i+2$ and $j'\leq i-2$,
\begin{equation}\label{eq : intgigj}\hat{\gamma}_{i}\pitchfork
  \hat{\gamma}_{j} \ \ \ \text{and}\ \ \  \gamma_{i}\pitchfork
  \hat{\gamma}_{j'}\end{equation} 
hold. Furthermore, the subsurface coefficient bound 
\begin{equation}\label{eq : dgigjgj'lb}d_{\hat{\gamma}_{i}}(\hat{\gamma}_{j},\hat{\gamma}_{j'})\geq e_{i-1}-C,\end{equation}
holds. 

Here $C=2B_{0}+7$ and $B_{0}$ is the constant from Theorem \ref{thm : beh} (Behrstock Inequality).  
\end{lem}
\begin{proof}
 Set the constant
$$E=C+B_{0}+G_{0}+2,$$
where $G_{0}$ is the constant from Theorem \ref{thm : bddgeod} for a geodesic in the curve complex of $S_{0,5}$.

 We prove (\ref{eq : intgigj}) and (\ref{eq : dgigjgj'lb}) simultaneously by induction on $j-j'$. The proof of the base of the induction breaks into the following cases:
 \medskip

\noindent{\bf Case $j-j'=4$, $j-i=2$ and $i-j'=2$:} Applying $(f_{1}\circ ...\circ
f_{i-2})^{-1}$ to the curves $\hat{\gamma}_{j'},\hat{\gamma}_{i}$ and
$\hat{\gamma}_{j}$, we obtain the curves
\begin{center}
$\hat{\gamma}_{0},\hat{\gamma}_{2}=f_{i-1}\circ f_{i}(\hat{\gamma}_{0})$ and
$f_{i-1}\circ ...\circ f_{i+2}(\hat{\gamma}_{0})$, 
\end{center}
respectively. The curve $f_{i-1}\circ ...\circ f_{i+2}(\hat{\gamma}_{0})$ is the same as $\hat{\gamma}_{4}$ with a different number of parallel strands, see Figure \ref{fig : nue05}. Since $\hat{\gamma}_{2}\pitchfork\hat{\gamma}_{0}$ and 
$\hat{\gamma}_{2}\pitchfork f_{i-1}\circ ...\circ f_{i+2}(\hat{\gamma}_{0})$ hold, it follows that
$\hat{\gamma}_{i}\pitchfork \hat{\gamma}_{i-2}$ and
$\hat{\gamma}_{i}\pitchfork \hat{\gamma}_{i+2}$ hold. This is (\ref{eq : intgigj}). 
  
  We proceed to establish (\ref{eq : dgigjgj'lb}). We have $f_{i-1}\circ ...\circ f_{i+2}(\hat{\gamma}_{0})=\mathcal{D}_{\hat{\gamma}_{2}}^{e_{i-1}}\circ \rho(\hat{\gamma}_{3})$. Then by the formula (\ref{eq : rtdt}) for the relative twists we have 
  $$\tau_{\hat{\gamma}_{2}}(f_{i-1}\circ ...\circ f_{i+2}(\hat{\gamma}_{0}),\rho(\hat{\gamma}_{3}))\subset \{e_{i-1},e_{i-1}+1\}.$$ 
  Then (\ref{eq : dYiYann}) implies that 
  $$d_{\hat{\gamma}_{2}}(f_{i-1}\circ ...\circ f_{i+2}(\hat{\gamma}_{0}),\rho(\hat{\gamma}_{3}))\geq e_{i-1}-3.$$
   Furthermore, the curves $\rho(\hat{\gamma}_{3})$ and $\hat{\gamma}_{0}$ are disjoint and both intersect $\hat{\gamma}_{2}$, thus 
   $$d_{\hat{\gamma}_{2}}(\hat{\gamma}_{0},\rho(\hat{\gamma}_{3}))\leq 1.$$
    Combining the above two subsurface coefficient bounds by the triangle inequality and using the fact that $\diam_{\hat{\gamma}_{2}}(\rho(\hat{\gamma}_{3}))\leq 1$, we have
\begin{equation}\label{eq : dg2>e}d_{\hat{\gamma}_{2}}(\hat{\gamma}_{0},f_{i-1}\circ ...\circ f_{i+2}(\hat{\gamma}_{i}))\geq e_{i-1}-5.\end{equation}
 Applying $f_{1}\circ ...\circ f_{i-2}$ to the subsurface coefficient above and using the fact that $C>5$, we obtain
  $$d_{\hat{\gamma}_{i}}(\hat{\gamma}_{i-2},\hat{\gamma}_{i+2})\geq e_{i-1}-C.$$
   This is the subsurface coefficient bound (\ref{eq : dgigjgj'lb}). 
   \medskip
    
  \noindent{\bf Case $j-j'=5$, $i-j'=2$ and $j-i=3$:} 
    Applying $(f_{1}\circ ...\circ f_{i-2})^{-1}$
  to the curves $\hat{\gamma}_{j'},\hat{\gamma}_{i}$ and
  $\hat{\gamma}_{j}$,  we obtain the curves 
  \begin{center}
  $\hat{\gamma}_{0},\hat{\gamma}_{2}=f_{i-1}\circ f_{i}(\hat{\gamma}_{0})$ and
$f_{i-1}\circ ...\circ f_{i+3}(\hat{\gamma}_{0})$,
\end{center}
 respectively. The curve $f_{i-1}\circ ...\circ f_{i+3}(\hat{\gamma}_{0})$ is the same as $\hat{\gamma}_{5}$ with a different number of parallel strands, see Figure \ref{fig : nue05}. Since $\hat{\gamma}_{0}\pitchfork \hat{\gamma}_{2}$ and $\hat{\gamma}_{2}\pitchfork f_{i-1}\circ ...\circ f_{i+3}(\hat{\gamma}_{0})$ hold, (\ref{eq : intgigj}) holds. 

We have $f_{i-1}\circ ...\circ f_{i+3}(\hat{\gamma}_{i})=\mathcal{D}_{\hat{\gamma}_{2}}^{e_{i-1}}\circ \rho (f_{i}\circ ...\circ f_{i+3}(\hat{\gamma}_{i}))$.
 Then by (\ref{eq : rtdt}), 
 $$\tau_{\hat{\gamma}_{2}}(f_{i-1}\circ ...\circ f_{i+3}(\hat{\gamma}_{i}),\rho\circ f_{i}\circ ...\circ f_{i+3}(\hat{\gamma}_{i}))\subset\{e_{i-1},e_{i-1}+1\}.$$ 
So (\ref{eq : dYiYann}) implies that 
$$d_{\hat{\gamma}_{2}}(f_{i-1}\circ ...\circ f_{i+3}(\hat{\gamma}_{i}),\rho\circ f_{i}\circ ...\circ f_{i+3}(\hat{\gamma}_{i}))\geq e_{i-1}-3.$$
 Furthermore,  because $\rho(\hat{\gamma}_{3})$ is a curve intersecting $\hat{\gamma}_{2}$ and disjoint from both $\hat{\gamma}_{0}$ and $\rho\circ f_{i}\circ ...\circ f_{i+3}(\hat{\gamma}_{i})$ (to see this, note that $f_{i}\circ...\circ f_{i+3}(\hat{\gamma}_{i})$ is $\hat{\gamma}_{4}$ with different number of parallel strands), we have
 $$d_{\hat{\gamma}_{2}}(\hat{\gamma}_{0},\rho\circ f_{i}\circ ...\circ f_{i+3}(\hat{\gamma}_{i}))\leq 2.$$
  Combining the above two subsurface coefficient bounds by the triangle inequality and using the fact that $\diam_{\hat{\gamma}_{2}}(\rho\circ f_{i}\circ ...\circ f_{i+3}(\hat{\gamma}_{i}))\leq 1$, we have 
$$d_{\hat{\gamma}_{2}}(\hat{\gamma}_{0},f_{i-1}\circ ...\circ f_{i+3}(\hat{\gamma}_{i}))\geq e_{i-1}-6.$$
Now applying $f_{1}\circ ...\circ f_{i-2}$ to the above subsurface coefficient and using the fact that $C>6$, we get   
  $$d_{\hat{\gamma}_{i}}(\hat{\gamma}_{i-2},\hat{\gamma}_{i+3})\geq e_{i-1}-C.$$
  This is the subsurface coefficient bound (\ref{eq : dgigjgj'lb}).
  \medskip
  
\noindent{\bf Case $j-j'=5$, $i-j'=3$ and $j-i=2$:} Applying $(f_{1}\circ ...\circ
  f_{i-3})^{-1}$ to the curves $\hat{\gamma}_{j'},\hat{\gamma}_{i}$
  and $\hat{\gamma}_{j}$, we obtain the curves 
  \begin{center}
  $\hat{\gamma}_{0},\hat{\gamma}_{3}=f_{i-2}\circ f_{i-1}\circ f_{i}(\hat{\gamma}_{0})$ and
$f_{i-2}\circ ...\circ f_{i+2}(\hat{\gamma}_{0})$, 
\end{center}
see Figure \ref{fig : nue05}. The statement about the intersection of curves (\ref{eq : intgigj}) holds since $\hat{\gamma}_{3}\pitchfork\hat{\gamma}_{0}$ and $\hat{\gamma}_{3}\pitchfork f_{i-2}\circ ...\circ f_{i+2}(\hat{\gamma}_{0})$. 

By the triangle inequality
\begin{eqnarray}\label{eq : tridg3g0g5}
d_{\hat{\gamma}_{3}}(\hat{\gamma}_{0},f_{i-2}\circ ...\circ f_{i+2}(\hat{\gamma}_{0}))&\geq& d_{\hat{\gamma}_{3}}(\hat{\gamma}_{1},f_{i-2}\circ ...\circ f_{i+2}(\hat{\gamma}_{0}))\\
&-&d_{\hat{\gamma}_{3}}(\hat{\gamma}_{0},\hat{\gamma}_{1})-\diam_{\hat{\gamma}_{3}}(\hat{\gamma}_{1}).\nonumber
\end{eqnarray}
First we find a lower bound for the first term on the right-hand side of (\ref{eq : tridg3g0g5}). Note that $f_{i-2}(\hat{\gamma}_{0})=\hat{\gamma}_{1}$ and $f_{i-2}(\hat{\gamma}_{2})=\hat{\gamma}_{3}$.
Thus applying $(f_{i-2})^{-1}$ to this term, we obtain
$$d_{\hat{\gamma}_{2}}(\hat{\gamma}_{0},f_{i-1}\circ ...\circ f_{i+2}(\hat{\gamma}_{0})).$$
This subsurface coefficient by (\ref{eq : dg2>e}) is bounded below by $e_{i-1}-5$.

The two curves $\hat{\gamma}_{0}$ and $\hat{\gamma}_{1}$ are disjoint and intersect $\hat{\gamma}_{3}$. So the second term on the right-hand side of (\ref{eq : tridg3g0g5}) is bounded by $1$.

These bounds for the two terms on the right-hand side of the inequality (\ref{eq : tridg3g0g5}) and the fact that $\diam_{\hat{\gamma}_{3}}(\hat{\gamma}_{1})\leq 1$ (Lemma \ref{lem : diamproj}) give us
$$d_{\hat{\gamma}_{3}}(\hat{\gamma}_{0},f_{i-2}\circ ...\circ f_{i+2}(\hat{\gamma}_{0}))\geq e_{i-1}-7> e_{i-1}-C.$$
Applying $f_{1}\circ...\circ f_{i-1}$ to the subsurface coefficient on the left-hand side of the above inequality, we obtain the bound (\ref{eq : dgigjgj'lb}).

We proved that (\ref{eq : intgigj}) and (\ref{eq : dgigjgj'lb}) hold for $j-j'\leq 5$. In what follows we assume that (\ref{eq : intgigj}) and (\ref{eq : dgigjgj'lb}) hold when $j-j'\leq n$, where $n\geq 6$, and prove that (\ref{eq : intgigj}) and (\ref{eq : dgigjgj'lb}) hold for $j-j'=n+1$
\medskip

 If $j-i=2$ or $3$, applying $(f_{1}\circ...\circ f_{i})^{-1}$ to $\hat{\gamma}_{i}$ and $\hat{\gamma}_{j}$ we obtain $\hat{\gamma}_{0}$ and $\hat{\gamma}_{j-i}$, respectively. Then since $\hat{\gamma}_{0}\pitchfork\hat{\gamma}_{j-i}$ (see Figure \ref{fig : nue05}), we have $\hat{\gamma}_{i}\pitchfork\hat{\gamma}_{j}$.  If $j-i=4$ or $5$, then since $(j-2)-i\geq 2$ and $j-(j-2)=2$, by the hypothesis of the induction
$$d_{\hat{\gamma}_{j-2}}(\hat{\gamma}_{i},\hat{\gamma}_{j})\geq E> 2.$$
This bound implies that $\hat{\gamma}_{i}\pitchfork \hat{\gamma}_{j}$ holds.

Now suppose that $j-i\geq 6$. Then we have $(j-2)-i\geq 2$. Thus by
the induction hypothesis  $\hat{\gamma}_{j-2}\pitchfork
\hat{\gamma}_{i}$ holds. Moreover, $\hat{\gamma}_{j}\pitchfork\hat{\gamma}_{j+2}$. So we may write the following triangle inequality 
\begin{eqnarray}\label{eq : dgj-2gigj3}
d_{\hat{\gamma}_{j-2}}(\hat{\gamma}_{i},\hat{\gamma}_{j})&\geq& d_{\hat{\gamma}_{j-2}}(\hat{\gamma}_{j-4},\hat{\gamma}_{j})-d_{\hat{\gamma}_{j-2}}(\hat{\gamma}_{j-4},\hat{\gamma}_{i})-\diam_{\hat{\gamma}_{j-2}}(\hat{\gamma}_{j-4})\nonumber\\
&\geq&E-C-B_{0}-1> 2.
\end{eqnarray}
To get the second inequality in (\ref{eq : dgj-2gigj3}), first, by the assumption of the induction, we have 
$$d_{\hat{\gamma}_{j-2}}(\hat{\gamma}_{j-4},\hat{\gamma}_{j})\geq e_{j-3}-C\geq E-C,$$
This gives a lower bound for the first term on the right-hand side of the first inequality of (\ref{eq : dgj-2gigj3}). Second, since $(j-4)-i\geq 2$ by the assumption of the induction we have 
$$d_{\hat{\gamma}_{j-4}}(\hat{\gamma}_{i},\hat{\gamma}_{j-2})\geq e_{j-5}-C>E-C> B_{0},$$
 where the last inequality holds because $E>C+B_{0}$.
Then Behrstock inequality (Theorem \ref{thm : beh} ) implies that
$$d_{\hat{\gamma}_{j-2}}(\hat{\gamma}_{j-4},\hat{\gamma}_{i})\leq B_{0}.$$ 
This is the upper bound for the second term on the right-hand side of the first inequality of (\ref{eq : dgj-2gigj3}). Finally the last term by Lemma \ref{lem : diamproj} is at most $1$.

The lower bound (\ref{eq : dgj-2gigj3}) guarantees that
$\hat{\gamma}_{i}\pitchfork \hat{\gamma}_{j}$ holds. The proof of that
$\hat{\gamma}_{j'}\pitchfork \hat{\gamma}_{i}$ holds for each $j'\leq i-2$ is
similar. The proof of (\ref{eq : intgigj}) is complete. 
 \medskip

We proceed to establish (\ref{eq : dgigjgj'lb}). Let $j,j'$ be so that $j'\leq i-2$ and $j\geq i+2$. By (\ref{eq : intgigj}) we may write the following triangle inequality 
\begin{eqnarray}\label{eq : dgigjgj'lbt}
d_{\hat{\gamma}_{i}}(\hat{\gamma}_{j'},\hat{\gamma}_{j})&\geq& d_{\hat{\gamma}_{i}}(\hat{\gamma}_{i-2},\hat{\gamma}_{i+2})-d_{\hat{\gamma}_{i}}(\hat{\gamma}_{i-2},\hat{\gamma}_{j'})-d_{\hat{\gamma}_{i}}(\hat{\gamma}_{i+2},\hat{\gamma}_{j})\nonumber\\
&-&\diam_{\hat{\gamma}_{i}}(\hat{\gamma}_{i-2})-\diam_{\hat{\gamma}_{i}}(\hat{\gamma}_{i+2}).
\end{eqnarray}
 We have that $(i-2)-j'<i-j'<j-j'$ and $j-(i+2)<j-i<j-j'$. Thus by the assumption of the induction, the fact that $e_i>E$ and the choice of $E$ we have that 
 \begin{eqnarray*}
 d_{\hat{\gamma}_{i-2}}(\hat{\gamma}_{i},\hat{\gamma}_{j'})&\geq& E-C> B_{0}, \;\text{and}\\
 d_{\hat{\gamma}_{i+2}}(\hat{\gamma}_{i},\hat{\gamma}_{j})&\geq& E-C> B_{0},
 \end{eqnarray*}
 The first lower bound above and the Behrstock inequality imply that the second term on the right-hand side of (\ref{eq : dgigjgj'lbt}) is bounded above by $B_{0}$. Similarly the second bound above and the Behrstock inequality imply that the third term on the right-hand side of (\ref{eq : dgigjgj'lbt}) is bounded above by $B_{0}$. Moreover, by Lemma \ref{lem : diamproj} the third and fourth terms on the right-hand side of (\ref{eq : dgigjgj'lbt}) are less than or equal to $1$. So we obtain
\begin{eqnarray*}
d_{\hat{\gamma}_{i}}(\hat{\gamma}_{j'},\hat{\gamma}_{j})&\geq& d_{\gamma_{i}}(\hat{\gamma}_{i-2},\hat{\gamma}_{i+2})-2B_{0}-2\\
&\geq&e_{i-1}-C.
\end{eqnarray*}
The proof of (\ref{eq : dgigjgj'lb}) is complete. 
\end{proof}

We proceed to prove the proposition. Part (\ref{nue05 : int}) is the statement about intersection of curves (\ref{eq : intgigj}) we proved in Lemma \ref{lem : int-subsurfbd}. Note that (\ref{eq : dgigjgj'lb}) gives the lower bound in part (\ref{nue05 : dgi}).  Part (\ref{nue05 : qg}) is Lemma 3.2 of \cite{nonuniqueerg}. Part (\ref{nue05 : dgi}) follows from parts (\ref{nue05 : int}), (\ref{nue05 : qg}) and Theorem \ref{thm : bddgeod} (Bounded Geodesics Image Theorem). 
\medskip

Now we prove part (\ref{nue05 : fill}) of the proposition. The proof is by induction on $j-i$ and is essentially the one given in Lemma 3.2 of \cite{nonuniqueerg}. Note that here we do not assume any upper bound for the value of $j-i$. 

In the rest of the proof denote the surafce $S_{0,5}$ by $S$. Suppose that $j-i=4$. Applying $(f_{1}\circ ...\circ f_{i})^{-1}$ to the curves $\hat{\gamma}_{i}$ and $\hat{\gamma}_{j}$ we obtain the curves $\hat{\gamma}_{0}$ and $\hat{\gamma}_{4}$ in Figure \ref{fig : nue05}, respectively, which fill $S$. Thus $\hat{\gamma}_{i}$ and $\hat{\gamma}_{j}$ fill $S$.

Suppose that part (\ref{nue05 : fill}) is true for all $j-i\leq n$, where $n\geq 5$. Let $j-i=n+1$. To get a contradiction suppose that curves $\hat{\gamma}_{i}$ and $\hat{\gamma}_{j}$ do not fill the surface. Then $d_{S}(\hat{\gamma}_{i},\hat{\gamma}_{j})\leq 2$. 
On the other hand, by the assumption of the induction the curves $\hat{\gamma}_{i}$ and $\hat{\gamma}_{j-1}$ fill $S$, so $d_{S}(\hat{\gamma}_{i},\hat{\gamma}_{j-1})\geq 3$. Moreover, by the construction of the sequence of curves  $\hat{\gamma}_{j}$ and $\hat{\gamma}_{j-1}$ are disjoint, so $d_{S}(\hat{\gamma}_{j},\hat{\gamma}_{j-1})=1$. Thus by the triangle inequality $d_{S}(\hat{\gamma}_{i},\hat{\gamma}_{j})\geq 2$. The two bounds we established for $d_{S}(\hat{\gamma}_{i},\hat{\gamma}_{j})$ imply that
$$d_{S}(\hat{\gamma}_{i},\hat{\gamma}_{j})= 2.$$
Since $j-i\geq 5$ we may choose an index 
\begin{center}$h$ so that $i<h<h+1<j$, $j-h-1\geq 2$ and $h-i\geq 2$.\end{center}
Then by (\ref{eq : intgigj}) the curves $\hat{\gamma}_{i}$ and $\hat{\gamma}_{j}$ intersect $\hat{\gamma}_{h}$ and $\hat{\gamma}_{h+1}$. Moreover, by the bound (\ref{eq : dgigjgj'lb}), the fact that $e_i>E$ and the choice of $E$ we have that
\begin{eqnarray*}
d_{\hat{\gamma}_{h}}(\hat{\gamma}_{i},\hat{\gamma}_{j})&\geq&E-C>G_{0},\;\text{and}\\
d_{\hat{\gamma}_{h+1}}(\hat{\gamma}_{i},\hat{\gamma}_{j})&\geq&E-C>G_{0}
\end{eqnarray*}
where $G_{0}$ is the constant from Theorem \ref{thm : bddgeod} for a geodesic in $\mathcal{C}(S)$. 

As we saw above $d_{S}(\hat{\gamma}_{j},\hat{\gamma}_{i})=2$, so the geodesic in $\mathcal{C}(S)$ connecting $\hat{\gamma}_{i}$ and $\hat{\gamma}_{j}$ contains three curves $\hat{\gamma}_{i},\gamma'$ and $\hat{\gamma}_{j}$. We have that the curve $\gamma'$ is disjoint from $\hat{\gamma}_{h}$. For otherwise, the curves $\gamma_{i},\gamma',\gamma_{j}$ which form a geodesic in $\mathcal{C}(S)$ intersect $\hat{\gamma}_{h}$. Then Theorem \ref{thm : bddgeod} (Bounded Geodesic Image) implies that $d_{\hat{\gamma}_{h}}(\hat{\gamma}_{i},\hat{\gamma}_{j})\leq G_{0}$. But this contradicts the first lower bound above. Similarly using the second lower bound above we may show that $\gamma'$ and $\hat{\gamma}_{h+1}$ are disjoint. The curves $\hat{\gamma}_{h}$ and $\hat{\gamma}_{h+1}$ consist a pants decomposition on $S$. Thus the only curves disjoint from both $\hat{\gamma}_{h}$ and $\hat{\gamma}_{h+1}$ are themselves. So $\gamma'$ is either $\hat{\gamma}_{h}$ or $\hat{\gamma}_{h+1}$. As we mentioned above the curves  $\hat{\gamma}_{h}$ and $\hat{\gamma}_{h+1}$ intersect the curves $\hat{\gamma}_{i}$ and $\hat{\gamma}_{j}$. So $\gamma'$ intersects both  $\hat{\gamma}_{i}$ and $\hat{\gamma}_{j}$. On the other hand, since $\hat{\gamma}_{i},\gamma'$ and $\hat{\gamma}_{j}$ are consecutive curves on a geodesic in $\mathcal{C}(S)$, $\gamma'$ is disjoint from both $\hat{\gamma}_{i}$ and $\hat{\gamma}_{j}$. This contradiction shows that in fact $\hat{\gamma}_i$ and $\hat{\gamma}_j$ fill $S$ and completes the proof of part (\ref{nue05 : fill}) by induction. 
\end{proof}

Let the sequence of integers $\{e_{i}\}_{i=1}^{\infty}$ with $e_{i}>E$, and the sequence of curves $\{\hat{\gamma}_{i}\}_{i=0}^{\infty}$ be as in Proposition \ref{prop : 1-nue05}. Part (\ref{nue05 : qg}) of Proposition \ref{prop : 1-nue05}  and hyperbolicity of the curve complex imply that the sequence of curves $\{\hat{\gamma}_{i}\}_{i=0}^{\infty}$ converges to a point in the Gromov boundary of the curve complex. By Proposition \ref{prop : bdrycc} this point determines a projective measured lamination $[\mathcal{E}]$ with minimal filling support $\hat{\nu}$ on $S_{0,5}$.  

\begin{prop}\label{prop : 2-nue05} Let the marking $\hat{\mu}$ and the geodesic lamination $\hat{\nu}$ be as above. We have
\begin{enumerate}
\item \label{nue05 : dgimn}There exist $K\geq 1$ and $C\geq 0$ so that $d_{\hat{\gamma}_{i}}(\hat{\mu},\hat{\nu})\asymp_{K,C} e_{i-1}$.

Furthermore, suppose that for some $a>2$ we have $e_{i+1}\geq ae_{i}$ for each $i\geq 1$. Then 
\item \label{nue05 : nue}The geodesic lamination $\hat{\nu}$ is minimal, filling and non-uniquely ergodic. 
\end{enumerate}
\end{prop}
Part (\ref{nue05 : dgimn}) follows from Proposition \ref{prop : 1-nue05} (\ref{nue05 : dgi})
and Theorem \ref{thm : bddgeod} (Bounded Geodesics Image Theorem). Part (\ref{nue05 : nue}) is Theorem 1.1 of
\cite{nonuniqueerg}. Note that the growth of powers $e_{i+1}\geq ae_{i}$ ($a>2$) is required
to guarantee the non-unique ergodicity of the lamination $\hat{\nu}$.  
\medskip

The utility of the construction of Leininger-Lenzhen-Rafi lies in its
control on  subsurface coefficients; see Proposition 
\ref{prop : 2-nue05} (\ref{nue05 : dgimn}), Theorem \ref{thm :
  nabddcomb} and Theorem \ref{thm : abd}. These are conditions similar to arithmetic conditions
for coefficients of the continued fraction expansion of irrational numbers
relevant to the coding of geodesics on the modular surface which is
$\mathcal{M}(S_{1,1})$ as well; see \cite{modsurf}. Though Gabai's
construction produces a minimal filling non-uniquely ergodic
lamination on any surface $S$ with $\xi(S)>1$, it provides no a priori
control on subsurface coefficients.  

\begin{thm}\label{thm : nabddcomb}
There is a constant $\hat{R}>0$ such that for any proper, essential, non-annular subsurface $Y\subsetneq S$ we have $d_{Y}(\hat{\mu},\hat{\nu})\leq\hat{R}$. In other words, the pair $(\hat{\mu},\hat{\nu})$ has
non-annular $\hat{R}-$bounded combinatorics.  
\end{thm}
\begin{proof} By Proposition \ref{prop : 1-nue05} (\ref{nue05 : qg}), $\{\hat{\gamma}_{i}\}_{i=0}^{\infty}$ is a $1-$Lipschitz, $(k,c)-$quasi-geodesic in $\mathcal{C}(S_{0,5})$. Let $G$ be the corresponding constant from Theorem \ref{thm :
  bddgeod} (Bounded Geodesic Image Theorem). 
Let $Y\subseteq S_{0,5}$ be an essential non-annular subsurface. First note that
$Y$ is a four-holed sphere.

 If $\hat{\gamma}_{i}\pitchfork Y$ holds for all
$i\geq 0$, then the Bounded Geodesic Image Theorem guarantees that 
$$\diam_{Y}(\{\hat{\gamma}_{i}\}_{i=0}^{\infty})\leq G.$$
  The lamination $\hat{\nu}$ is filling, so $\pi_{Y}(\hat{\nu})\neq\emptyset$. Now we claim that 
  \begin{equation}\label{eq : dYmn}d_{Y}(\hat{\mu},\hat{\nu})\leq G+6.\end{equation}
  To see this, let $\hat{\gamma}_{j_{n}}$ be a convergent subsequence of $\{\hat{\gamma}_{j}\}_{j=i+2}^{\infty}$ in the $\mathcal{PML}(S)$ topology. By Proposition \ref{prop : bdrycc} the support of the limit of $\hat{\gamma}_{j_{n}}$ is $\hat{\nu}$. After possibly passing to a further subsequence we may assume that $\hat{\gamma}_{j_{n}}$ is also convergent in the Hausdorff topology of $M_{\infty}(S)$. Denote the limit lamination in the Hausdorff topology by $\xi$.  Then $\hat{\nu}\subseteq \xi$ (see e.g. \cite{notesonthurston}). By the bound $\diam_{Y}(\{\hat{\gamma}_{i}\}_{i= 0}^{\infty})\leq G$ we established above, we have that $d_{Y}(\hat{\gamma}_{0},\hat{\gamma}_{j_{n}})\leq G$. Then since $\hat{\gamma}_{j_{n}}\to\xi$ in the Hausdorff topology as $n\to\infty$, by Lemma \ref{lem : subsurfcomplim}, we obtain 
  $$d_{Y}(\xi,\hat{\gamma}_{0})\leq G+4.$$ 
 Furthermore we have that $\hat{\nu}\subseteq\xi$ and $\hat{\gamma}_0\subset\hat{\mu}$. Then since $\diam_{Y}(\xi)\leq 2$ (by Lemma \ref{lem : diamproj}), the difference of the subsurface projection distance in (\ref{eq : dYmn}) and the one above is most at $2$. Which gives us (\ref{eq : dYmn}).
  \medskip
  
  Now suppose that for some integer $i\geq 0$, $\hat{\gamma}_{i}\pitchfork Y$ does not hold. Then since $Y$ is a four-holed sphere inside $S_{0,5}$ we
have that $\partial{Y}=\hat{\gamma}_{i}$.  

Let $j'\leq i-2$. By part (\ref{nue05 : int}) of Proposition \ref{prop : 1-nue05}, $\hat{\gamma}_{i}\pitchfork \hat{\gamma}_{j'}$. Then since $\partial{Y}=\hat{\gamma}_{i}$,
we conclude that $\hat{\gamma}_{j'}\pitchfork Y$ holds. Thus Bounded Geodesic Image
Theorem guarantees that 
$$\diam_{Y}(\{\hat{\gamma}_{j'}\}_{j'=0}^{i-2})\leq G.$$ 
The above bound and the fact that $\hat{\mu}$ contains
$\hat{\gamma}_{0}$ give us the bound  
\begin{equation}d_{Y}(\hat{\mu},\hat{\gamma}_{i-2})\leq G.\label{eq  : dYgi-m}\end{equation}
Let $j\geq i+2$. By part (\ref{nue05 : int}) of Proposition \ref{prop : 1-nue05}, $\hat{\gamma}_{i}\pitchfork \hat{\gamma}_{j}$ holds. Then similarly to above we obtain that 
$$\diam_{Y}(\{\hat{\gamma}_{j}\}_{j= i+2}^{\infty})\leq G.$$
  Then similar to the proof of (\ref{eq : dYmn}) we may obtain
\begin{equation}d_{Y}(\hat{\nu},\hat{\gamma}_{i+2})\leq G+6. \label{eq  : dYgi+n} \end{equation}

By Proposition \ref{prop : 1-nue05} (\ref{nue05 : dgi}) $\hat{\gamma}_{i-2}\pitchfork\hat{\gamma}_{i}$. So $\hat{\gamma}_{i-2}\pitchfork
Y$ holds, because $\hat{\gamma}_{i}=\partial{Y}$. Similarly $
\hat{\gamma}_{i+2}\pitchfork Y$ holds. So
$d_{Y}(\hat{\gamma}_{i-2},\hat{\gamma}_{i+2})$ is defined. We claim that  
\begin{equation}\label{eq : dYgi-+}d_{Y}(\hat{\gamma}_{i-2},\hat{\gamma}_{i+2})=1.\end{equation} 
To see this, let $g$ be the element of $\Mod(S)$ given by the composition $g=f_{1}\circ ...\circ
f_{i-2}$. Applying $g^{-1}$ to the subsurface
coefficient in (\ref{eq : dYgi-+}) we get
$$d_{g^{-1}(Y)}(g^{-1}(\hat{\gamma}_{i-2}),g^{-1}(\hat{\gamma}_{i+2})).$$ 
Thus, to obtain the desired equality, it suffices to show that the above
subsurface coefficient is equal to $1$. The curves 
$$g^{-1}(\hat{\gamma}_{i-2}),...,g^{-1}(\hat{\gamma}_{i+2})$$ are the
curves $\hat{\gamma}_{0},...,\hat{\gamma}_{4}$ in Figure \ref{fig :
  nue05}, respectively, except that the twist of the curve $g^{-1}(\hat{\gamma}_{i+2})$ about
$g^{-1}(\hat{\gamma}_{i})=\hat{\gamma}_{2}$ is $e_{i-1}$ rather than $e_{1}$. 

\begin{figure}
\centering
\includegraphics[scale=.19]{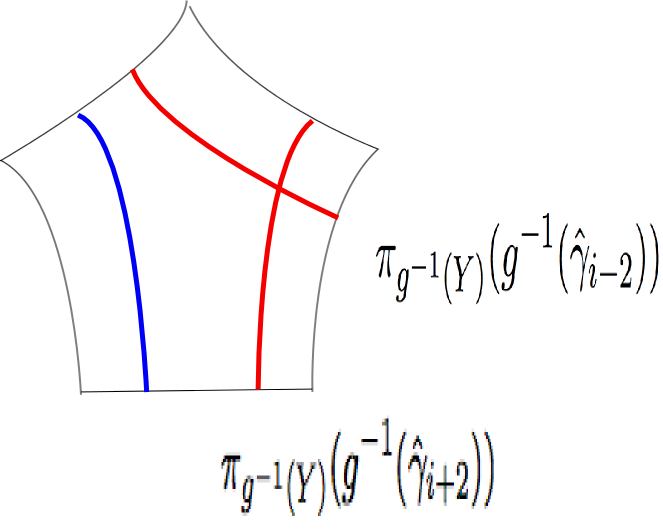}
\caption{The subsurface $g^{-1}(Y)\subset S_{0,5}$ is the four-holed sphere with boundary $\partial g^{-1}(Y)=\hat{\gamma}_{2}$. The curves $\pi_{g^{-1}(Y)}(g^{-1}(\hat{\gamma}_{i-2}))$ and $\pi_{g^{-1}(Y)}(g^{-1}(\hat{\gamma}_{i+2}))$ are shown in the figure.} 
\label{fig : projg-1Y}
\end{figure}

Since $\partial{Y}=\hat{\gamma}_{i}$, the subsurface $g^{-1}(Y)$ is the four holed sphere with
boundary $\hat{\gamma}_{2}$, see Figure \ref{fig : projg-1Y}. We have that
$g^{-1}(\hat{\gamma}_{i-2})=\hat{\gamma}_{0}$. Furthermore, the curve $g^{-1}(\hat{\gamma}_{i+2})$ is the curve
$\hat{\gamma}_{4}$ in Figure \ref{fig : nue05}, except that the twist of the curve $g^{-1}(\hat{\gamma}_{i+2})$ about
$\hat{\gamma}_{2}$ is $e_{i-1}$ rather than $e_{1}$.  The
projection of the curves $g^{-1}(\hat{\gamma}_{i-2})$ and $g^{-1}(\hat{\gamma}_{i+2})$ to the subsurface $g^{-1}(Y)$
are shown in Figure \ref{fig : projg-1Y}. The $\mathcal{C}(g^{-1}(Y))-$distance of these two curves is
$1$, because these are two curves with (minimal) intersection number $2$
on the four-holed sphere $g^{-1}(Y)$, yielding the desired equality.

Note that $\hat{\mu}$ is a marking and $\hat{\nu}$ fills the surface. By the triangle inequality and the bounds (\ref{eq  : dYgi-m}), (\ref{eq : dYgi+n}) and (\ref{eq : dYgi-+}) we have
\begin{eqnarray*}
d_{Y}(\hat{\mu},\hat{\nu})&\leq& d_{Y}(\hat{\mu},\hat{\gamma}_{i-2})+d_{Y}(\hat{\gamma}_{i-2},\hat{\gamma}_{i+2})+d_{Y}(\hat{\gamma}_{i+2},\hat{\nu})\\
&+&\diam_{Y}(\hat{\gamma}_{i-2})+\diam_{Y}(\hat{\gamma}_{i+2})\\
&\leq& 2G+6+1+4.
\end{eqnarray*}
We conclude that the $Y$ subsurface coefficient of $\hat{\mu}$ and $\hat{\nu}$ is bounded above by $\hat{R}:=2G+11$, as was desired.
\end{proof}

\begin{thm}\label{thm : abd}
There is a constant $R'>0$, so that $d_{\beta}(\hat{\mu},\hat{\nu})\leq R'$ for any curve $\beta$ which is not in the sequence $\{\hat{\gamma}_{i}\}_{i=0}^{\infty}$.
\end{thm}
\begin{proof}
By Proposition \ref{prop : 1-nue05} (\ref{nue05 : qg}), $\{\hat{\gamma}_{i}\}_{i=0}^{\infty}$ is a $1-$Lipschitz, $(k,c)-$quasi-geodesic in $\mathcal{C}(S_{0,5})$. Let $G$ be the corresponding constant from Theorem \ref{thm :
  bddgeod} (Bounded Geodesic Image Theorem).

If $\beta$ intersects all of the curves in the sequence. Then similar to (\ref{eq : dYmn}) in the proof of Theorem \ref{thm : nabddcomb} we may obtain that
$$d_{\beta}(\hat{\mu},\hat{\nu})\leq G+5.$$ 
Now suppose that for some integer $i>0$ the curve $\beta$ is disjoint from $\hat{\gamma}_i$. By Proposition 3.1(2)  for every $j\geq i+4$, the curves $\hat{\gamma}_j$ and $\hat{\gamma}_i$ fill $S_{0,5}$. So we may deduce that $\beta\pitchfork\hat{\gamma}_j$ holds. Then Theorem \ref{thm : bddgeod} guarantees that
\begin{equation}\label{eq : dbngi-}d_{\beta}(\hat{\mu},\hat{\gamma}_{i-4})\leq G.\end{equation}
Similarly for every $j'\leq i-4$, $\beta\pitchfork\hat{\gamma}_{j'}$ holds. Then 
$$\diam_{\beta}(\{\gamma_{j'}\}_{j'=0}^{ i-4}).$$
by the Bounded Geodesic Image Theorem. Then similarly to (\ref{eq : dYmn}) we may obtain
\begin{equation}\label{eq : dbngi+}d_{\beta}(\hat{\nu},\hat{\gamma}_{i+4})\leq G+5.\end{equation}
Let $g=f_{1}\circ ...\circ f_{i}$. Applying $g^{-1}$ to the curves $\hat{\gamma}_{i},...,\hat{\gamma}_{i+4}$, we obtain the curves $\hat{\gamma}_{0},...,\hat{\gamma}_{4}$ in Figure \ref{fig : nue05}, respectively. The difference is that $g^{-1}(\hat{\gamma}_{i+4})$ has $e_{i+1}$ twists. The only curve disjoint from $\hat{\gamma}_{0}$ and $\hat{\gamma}_{2}$ is $\hat{\gamma}_{1}$. Therefore, the only curve disjoint from $\hat{\gamma}_{i}$ and $\hat{\gamma}_{i+2}$ is $\hat{\gamma}_{i+1}$. The curve $\beta$ is not in the sequence $\{\hat{\gamma}_{i}\}_{i=0}^{\infty}$, in particular $\beta\neq \hat{\gamma}_{i+1}$. Moreover $\beta$ is disjoint from $\hat{\gamma}_{i}$. Thus $\beta\pitchfork \hat{\gamma}_{i+2}$ holds. Furthermore, the curves $\hat{\gamma}_{0}$ and $g^{-1}(\hat{\gamma}_{i+4})$ fill $S$ ($g^{-1}(\hat{\gamma}_{i+4})$ is $\hat{\gamma}_{4}$ with a different number of parallel strands). Thus $\hat{\gamma}_{i}$ and $\hat{\gamma}_{i+4}$ fill $S$. Then since $\beta$ is disjoint from $\hat{\gamma}_{i}$, $\beta\pitchfork\hat{\gamma}_{i+4}$ holds.
We showed that $\beta\pitchfork\hat{\gamma}_{i+2}$ and $\beta\pitchfork\hat{\gamma}_{i+4}$, therefore $d_{\beta}(\hat{\gamma}_{i+2},\hat{\gamma}_{i+4})$ is defined. Now since $i(\hat{\gamma}_{i+2},\hat{\gamma}_{i+4})=2$, by (\ref{eq : dYi}) we obtain the bound
\begin{equation}\label{eq : dbgi+}d_{\beta}(\hat{\gamma}_{i+2},\hat{\gamma}_{i+4})\leq 5.\end{equation}
Similarly, we may obtain the bound
 \begin{equation}\label{eq : dbgi-}d_{\beta}(\hat{\gamma}_{i-2},\hat{\gamma}_{i-4})\leq 5.\end{equation}
Let $g=f_{1}\circ ...\circ f_{i-2}$. Applying $g^{-1}$ to the curves $\hat{\gamma}_{i-2},...,\hat{\gamma}_{i+2}$ we obtain the first five curves in Figure \ref{fig : nue05}, with the difference that the last curve has $e_{i-1}$ twists. Let $Y\subset S_{0,5}$ be the the four-holed sphere with boundary curve $\hat{\gamma}_{i}$. Then $\beta\in \mathcal{C}_{0}(Y)$. The curves $\pi_{g^{-1}(Y)}(g^{-1}(\hat{\gamma}_{i-2}))$ and $\pi_{g^{-1}(Y)}(g^{-1}(\hat{\gamma}_{i+2}))$ are shown in Figure \ref{fig : projg-1Y}. These two curves intersect twice. Thus by (\ref{eq : dYi}) we have
   \begin{equation}\label{eq : dbgi+-}d_{\beta}(\hat{\gamma}_{i-2},\hat{\gamma}_{i+2})\leq 5.\end{equation}
  The bounds (\ref{eq : dbngi-}), (\ref{eq : dbngi+}), (\ref{eq : dbgi+}), (\ref{eq : dbgi-}) and (\ref{eq : dbgi+-}) for the $\beta$ subsurface coefficients combined by the triangle inequality give us the bound $R':=2G+29$. 
\end{proof}

We proceed to construct minimal, filling, non-uniquely ergodic
laminations on any surface $S_{g,0}$ of genus $g\geq 2$ with control on
the subsurface coefficients of the laminations. We construct the
laminations by an appropriate lift of the lamination we described on
$S_{0,5}$ using 2-dimensional orbifolds and their orbifold
covers. Here, we replace each puncture with a marked point
on the surface. 

Let $\mathcal{S}_{0,5}$ be the $2$ sphere equipped with an orbifold
structure with five orbifold points of order $2$ at the $5$ marked points of
$S_{0,5}$. Let $\mathcal{S}_{0,6}$ be the 2 sphere with an orbifold
structure with orbifold points of order $2$ at the marked points of $S_{0,6}$. Let $\mathcal{S}_{g,0}$ ($g\geq 2$) be $S_{g,0}$ equipped
with an orbifold structure with no orbifold point ( i.e. a manifold structure).  Let
$$f:\mathcal{S}_{0,6}\to \mathcal{S}_{0,5} \ \ \ \text{and} \ \ \  h:\mathcal{S}_{2,0}\to
\mathcal{S}_{0,6}$$ 
be the orbifold covering maps shown at the top left
and right of Figure~\ref{fig : sgs2}, respectively. Given $g\geq 2$,
let $\sigma_{g}:\mathcal{S}_{g,0}\to \mathcal{S}_{2,0}$ be the
covering map given at the bottom of Figure \ref{fig : sgs2}. Let
$F_{g}=\sigma_{g}\circ h\circ f$. Let $\nu$ be the lamination
$\nu=F_{g}^{-1}(\hat{\nu})$.

Recall the sequence of curves $\{\hat{\gamma}_{i}\}_{i=0}^{\infty}$. Denote the surface $S_{g,0}$ by $S$ and the covering map $F_{g}$ by $F$.

\begin{thm}\label{thm : g0qg}
There are constants $k\geq 1$ and $c\geq 0$ so that the sequence $\{F^{-1}(\hat{\gamma_{i}})\}_{i=0}^{\infty}$ is a $(k,c)-$quasi-geodesic in $\mathcal{C}(S)$.
\end{thm}
\begin{proof}
By Theorem 8.1 of \cite{coverscc} we have that the set-valued map that assigns to each simple closed curve $\alpha$ on the orbifold $\mathcal{S}_{0,5}$ the component curves of $F^{-1}(\alpha)$ in the orbifold cover $\mathcal{S}_{g,0}$ is a $(Q,Q)-$quasi-isometry from
$\mathcal{C}(S_{0,5})$ to $\mathcal{C}(S_{g,0})$, where $Q\geq 1$ is a
constant depending only on the degree of the cover $4(g-1)$. Then the theorem follows from Proposition \ref{prop : 1-nue05} (\ref{nue05 : qg}).
\end{proof}

For each $i\geq 0$ let $\gamma_{i}$ be a component curve of $F^{-1}(\hat{\gamma}_{i})$. By Theorem \ref{thm : g0qg} we have 
$$d_{S}(\gamma_{i},\gamma_{j})\geq\frac{1}{k}|i-j|-c.$$
Let $d=2k+kc$. If $|i-j|\geq d$, then by the above inequality we have
$$d_{S}(\gamma_{i},\gamma_{j})\geq 2,$$
which implies that
\begin{equation}\label{eq : intgigjcomp}\gamma_{j}\pitchfork \gamma_{i}\end{equation}
 holds.

\begin{thm}\label{thm : nueg0}Let the sequence of curves $\{\gamma_{i}\}_{i=0}^{\infty}$, the marking $\mu$ and the lamination $\nu$ be as above. There are constants $K\geq 1$ and $C\geq 0$ depending only on the degree of the cover such that we have
\begin{enumerate}

\item \label{nueg0 : dgi} For any $i\geq d$, $j\geq i+d$ and $j'\leq i-d$ we have $d_{\gamma_{i}}(\gamma_{j},\gamma_{j'})\asymp_{K,C} e_{i-1}$.
\item  \label{nueg0 : dgimn}$d_{\gamma_{i}}(\mu,\nu)\asymp_{K,C} e_{i-1}$ for all $i\geq 1$.
\item \label{nueg0 : nabd} For any essential, non-annular subsurface $W$
  we have $d_{W}(\mu,\nu)\leq R$. 
\item \label{nueg0 : nue} The lamination $\nu$ is a minimal filling non-uniquely
  ergodic lamination on $S_{g,0}$. 
\end{enumerate}
\end{thm}
\begin{proof}

First we prove part (\ref{nueg0 : nue}). Since $\hat{\nu}$ is a
non-uniquely ergodic lamination there are curves $\alpha,\beta\in
\mathcal{C}_{0}(S_{0,5})$ and measures $\hat{m}$ and $\hat{m}'$ supported on
$\hat{\nu}$ so that 
$$\frac{\hat{m}(\alpha)}{\hat{m}'(\alpha)}\neq\frac{\hat{m}(\beta)}{\hat{m}'(\beta)}.$$ 
Let $m=F^{*}(\hat{m})$ and
$m'=F^{*}(\hat{m})$ be the pull-backs of $\hat{m}$ and $\hat{m}'$,
respectively. Then $m$ and $m'$ are measures supported
on $\nu$. Let $\tilde{\alpha}$ be a  component of $F^{-1}(\alpha)$
and $\tilde{\beta}$ be a component of $F^{-1}(\beta)$. Then $m(\tilde{\alpha})=\hat{m}(\alpha)$ and $m(\tilde{\beta})=\hat{m}(\beta)$. Therefore
 $$\frac{m(\tilde{\alpha})}{m'(\tilde{\alpha})}\neq\frac{m(\tilde{\beta})}{m'(\tilde{\beta})},$$ 
and hence the lamination
$\nu$ is a non-uniquely ergodic lamination.  

We proceed to show that the lamination $\nu$ is minimal and filling. We use the facts
stated in $\S$\ref{subsec : cc} about measured laminations and foliations and the
correspondence between them. Equip $\hat{\nu}$ with a transverse
measure $\hat{m}$ and $\nu$ with measure the measure $m=F^{*}(\hat{m})$. Let $(\hat{\mathcal{F}},\hat{m})$ and $(\mathcal{F},m)$ be
the measured foliations corresponding to $(\hat{\nu},\hat{m})$ and
$(\nu,m)$, respectively. Note that $\mathcal{F}=F^{-1}(\hat{\mathcal{F}})$. Since the lamination $\hat{\nu}$ is minimal, the foliation $\hat{\mathcal{F}}$
is minimal. By the result of Hubbard and Masur \cite{hubmas} given a
complex structure on the surface $S_{0,5}$ there is a unique quadratic
differential $\hat{q}$ with vertical measured foliation
$(\hat{\mathcal{F}},\hat{m})$. Then $(\mathcal{F},m)$ is the vertical measured foliation of the quadratic differential $q=F^{*}(\hat{q})$. Since $\mathcal{F}$ is a minimal foliation on
$\mathcal{S}_{0,5}$ any leaf of $\mathcal{F}$ is dense in the
surface. Therefore, the lift of each leaf of $\mathcal{F}$ to
$\mathcal{S}_{g,0}$ is dense. To see this, let $l$ be a leaf of
$\mathcal{F}$. Suppose to the contrary that $l$ misses an open set $U$ in
$\mathcal{S}_{g,0}$. We may shrink $U$ and assume that the restriction
of $F$ to $U$ is a homeomorphism. But then $F(l)$ which is a leaf of
$\hat{\mathcal{F}}$ misses $F(U)$, which is an open subset of
$\mathcal{S}_{0,5}$. This contradicts the fact that
$\hat{\mathcal{F}}$ is a minimal foliation of $S_{0,5}$. Therefore $\mathcal{F}$ is minimal and consequently $\nu$ is as well. 

To see that the lamination $\nu$ fills $S$, note that given $\alpha\in
\mathcal{C}_{0}(S)$, a homotopy that realizes $\alpha$ and $\nu$ as
disjoint subsets of $\mathcal{S}_{g,0}$ composed with $F$ gives us a
homotopy which realizes $F(\alpha)$ (an essential closed curve on
$S_{0,5}$) and $\hat{\nu}$ as disjoint subsets of $S_{0,5}$. But this
contradicts the fact that $\hat{\nu}$ fills $S_{0,5}$.  
\medskip  

Using the terminology of \cite{coverscc} we say that a
subsurface $W\subseteq S_{g,0}$ is a symmetric subsurface if it is a
component of $F^{-1}(Y)$ for some subsurface $Y\subseteq
S_{5,0}$. 

When the subsurface $W$ is not symmetric by Lemma 7.2 of \cite{coverscc},
we have 
\begin{equation}\label{eq : dWnsym}d_{W}(\mu,\nu)\leq 2T_{e}+1,\end{equation}
 for a constant $T_{e}>0$
depending only on the degree of the cover and the constant $e$ which
comes from Rafi's characterization of short curves along
Teichm\"{u}ller geodesics; see $\S 4$ of \cite{coverscc} and for more
detail \cite{rshteich}, \cite{rteichhyper}. 

When the subsurface $W$ is an essential symmetric subsurface we have
$$d_{W}(\mu,\nu)\leq d_{Y}(\hat{\mu},\hat{\nu})$$ 
(see the proof of Theorem 8.1 in \cite{coverscc}). Furthermore, by Theorem \ref{thm : nabddcomb}, we know
that there exists $\hat{R}>0$ so that
$$d_{Y}(\hat{\mu},\hat{\nu})\leq \hat{R}$$
 for every essential, non-annular subsurface $Y\subseteq S_{0,5}$. The above two inequalities for subsurface coefficients give us
 \begin{equation}\label{eq : dWsym} d_{W}(\mu,\nu)\leq \hat{R}.\end{equation}
 Then by the subsurface coefficient
bounds (\ref{eq : dWnsym}) and  (\ref{eq : dWsym}) we obtain the upper bound
$R:=\max\{\hat{R},2T_{e}+1\}$ in part (\ref{nueg0 : nabd}).
 
We proceed to prove parts (\ref{nueg0 : dgi}) and (\ref{nueg0 : dgimn}). The fact that the subsurface coefficient $d_{\gamma_{i}}(\gamma_{j},\gamma_{j'})$ in part (\ref{nueg0 : dgi}) is defined follows from (\ref{eq : intgigjcomp}). Note that each annular subsurface with core curve $\gamma_{i}$ is a
 symmetric subsurface, because $\gamma_{i}$ is a component of
 $F^{-1}(\hat{\gamma}_{i})$. Thus as is shown in the proof of Theorem
 8.1 in \cite{coverscc} there exists $Q\geq 1$ so that 
 \begin{eqnarray}
 d_{\gamma_{i}}(\gamma_{j},\gamma_{j'})&\asymp_{Q,Q}&d_{\hat{\gamma}_{i}}(\hat{\gamma}_{j},\hat{\gamma}_{j'}),\;\text{and}\label{eq : dgigjgj'c} \\
 d_{\gamma_{i}}(\mu,\nu)&\asymp_{Q,Q}& d_{\hat{\gamma_{i}}}(\hat{\nu},\hat{\mu})\label{eq : dgimnc}.
 \end{eqnarray} 
 By Proposition \ref{prop : 1-nue05} (\ref{nue05 : dgi}) we have 
 $$d_{\hat{\gamma}_{i}}(\hat{\gamma}_{j},\hat{\gamma}_{j'})\asymp e_{i-1},$$ 
 then from the quasi-equality of subsurface coefficients (\ref{eq : dgigjgj'c}) the quasi-equality (\ref{nueg0 : dgi}) follows.
 
Moreover, by Proposition \ref{prop : 2-nue05} (\ref{nue05 : dgi}) we have
 $$d_{\hat{\gamma}_{i}}(\hat{\mu},\hat{\nu})\asymp e_{i-1},$$ 
  then from the quasi-equality of subsurface coefficients (\ref{eq :
   dgimnc}) the quasi-equality (\ref{nueg0 : dgimn}) follows.
 \end{proof}

\begin{figure}
\centering
\includegraphics[scale=.19]{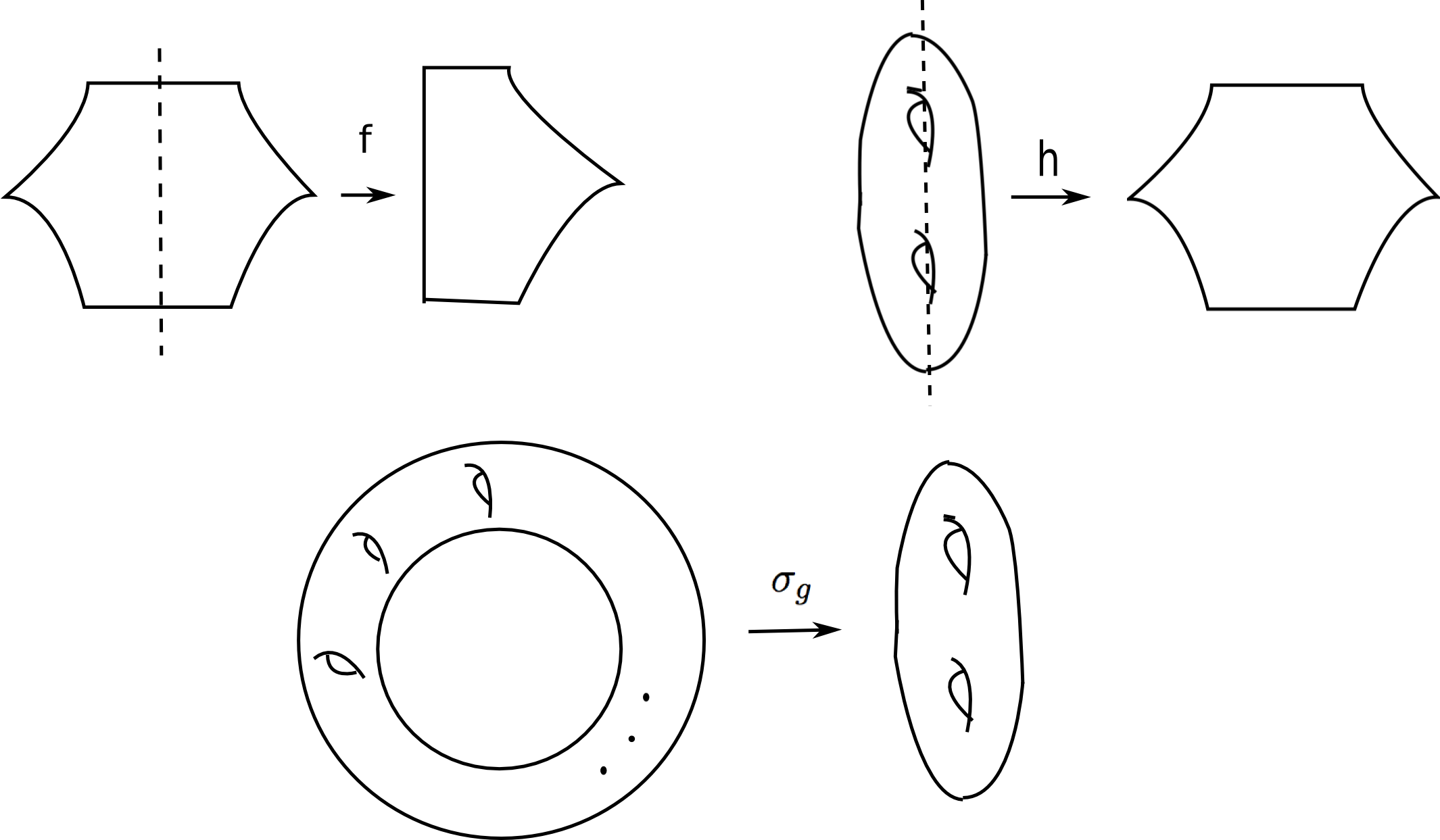}
\caption{Top left: $f:\mathcal{S}_{0,6}\to \mathcal{S}_{0,5}$ is the
  orbifold covering map realized as the rotation by angle $\pi$ about
  the axis in the picture. Top right: $h:\mathcal{S}_{2,0}\to
  \mathcal{S}_{0,6}$ is the orbifold covering map realized as the
  rotation by angle $\pi$ about the axis in the picture (hyperelliptic
  involution). Bottom: $\sigma_{g}:\mathcal{S}_{g,0}\to
  \mathcal{S}_{2,0}$ is the covering map realized by $g-1$ times
  iteration of the rotation by angle $\frac{2\pi}{g-1}$.} 
\label{fig : sgs2}
\end{figure}

\section{Recurrence of geodesics}\label{sec : recur}

Let $X\in \Teich(S)$ and $\nu^{-}$ be a Bers marking of $X$. Let
$\nu^{+}$ be a minimal filling lamination. By Lemma \ref{lem : rayinfty} there is an infinite WP geodesic ray
$r:[0,\infty)\to \Teich(S)$ with $r(0)=X$ and end invariant
$(\nu^{-},\nu^{+})$. Denote the projection of $r$ to the moduli space
by $\hat{r}$. For $\epsilon>0$, the $\epsilon-$thick part of $\mathcal{M}(S)$ consists of the Riemann surfaces with injectivity radius greater than $\epsilon$.

\begin{thm}\label{thm : recurrence}
Given $R>0$. Suppose that the pair $(\nu^{-},\nu^{+})$ has non-annular
$R-$bounded combinatorics. There is an $\epsilon>0$ such that $\hat{r}$ is recurrent to the $\epsilon-$thick part of the moduli space.
\end{thm}

 \begin{lem}\label{lem : zetamnthick}
Given $R>0$, there are constants $T_{0}>0$ and $\epsilon>0$  with the following property. Let $T>T_{0}$ and $\zeta_{n}:[0,T]\to\Teich(S)$ be a sequence of WP geodesic segments parametrized by arc-length so that the pair $(Q(\zeta_{0}(0)),Q(\zeta_{n}(T)))$ has non-annular $R-$bounded combinatorics. Then there is a time $t^{*}\in[0,T]$ and a sequence $\{m_{n}\}_{n=1}^{\infty}$ so that for all $n$ sufficiently large we have
$$\inj(\zeta_{m_{n}}(t^{*}))>\epsilon.$$
\end{lem}
\begin{proof}
Consider the limiting picture of geodesic segments $\zeta_{n}$ as was described in
 Theorem \ref{thm : geodlimit}. Let the partition $0=t_{0}<...<t_{k+1}=T$, the multi-curves $\sigma_{i}$, $i=0,...,k+1$, the multi-curve $\hat{\tau}$, and the piecewise geodesic 
 $$\hat{\zeta}:[0,T]\to \overline{\Teich(S)}$$
  be as in the theorem. Furthermore, recall the elements of the mapping class group $\psi_{n}$ for $n\in\mathbb{N}$, and $\mathcal{T}_{i,n}\in\tw(\sigma_{i}-\hat{\tau})$ for $i=0,...,k+1$ and $n\in\mathbb{N}$. Also as in the theorem set $\varphi_{i,n}=\mathcal{T}_{i,n}\circ ...\circ\mathcal{T}_{1,n}\circ\psi_{n}$.

 First we show that $\hat{\tau}=\emptyset$. Suppose to the contrary that $\hat{\tau}\neq \emptyset$. 

From Theorem \ref{thm : geodlimit} we know that
$\hat{\tau}=\sigma_{0}\cap\sigma_{1}$, so $\hat{\tau}\subseteq
\sigma_{0}$. Moreover by Theorem \ref{thm : geodlimit} (\ref{gl :
  sigma}), $\hat{\zeta}(0)\in \mathcal{S}(\sigma_{0})$. Thus for any
$\alpha\in \hat{\tau}$ we have $\ell_{\alpha}(\hat{\zeta}(0))=0$.
Furthermore, by Theorem \ref{thm : geodlimit} (\ref{gl : converge}), after possibly passing to a subsequence 
$\psi_{n}(\zeta_{n}(0))\to \hat{\zeta}(0)$ as $n\to\infty$. Thus by continuity of length functions for all $n$
sufficiently large and any $\alpha\in \hat{\tau}$,
$\ell_{\alpha}(\psi_{n}(\zeta_{n}(0))\leq L_{S}$. Thus there is
$Q_{0,n}$ a Bers pants decomposition of $\psi_{n}(\zeta_{n}(0))$
that contains $\hat{\tau}$. 
 
Similarly since $\hat{\tau}=\sigma_{k}\cap\sigma_{k+1}$ (as in Theorem \ref{thm : geodlimit}), we have
$\hat{\tau}\subseteq \sigma_{k+1}$. Moreover, by Theorem \ref{thm : geodlimit}(\ref{gl : sigma}), we know that $\hat{\zeta}(T)\in\mathcal{S}(\sigma_{k+1})$. Thus for any $\alpha\in\hat{\tau}$,
we have $\ell_{\alpha}(\hat{\zeta}(T))=0$. Furthermore, by Theorem
\ref{thm : geodlimit} (\ref{gl : converge}) after possibly passing to a subsequence, $\varphi_{k,n}(\zeta_{n}(T))\to\hat{\zeta}(T)$ as $n\to \infty$. Thus by continuity of length functions for all $n$ sufficiently large
and any $\alpha\in \hat{\tau}$,
we have that $\ell_{\alpha}(\varphi_{k,n}(\zeta_{n}(T))\leq L_{S}$. 

Now note that we have
$\varphi_{k,n}=\mathcal{T}_{k,n}\circ ...\circ\mathcal{T}_{1,n}\circ
\psi_{n}$. The element of mapping class group $\mathcal{T}_{i,n}$ is a composition of powers of Dehn
twists about the curves in $\sigma_{i}$ and
$\hat{\tau}\subseteq \sigma_{i}$. Therefore, $\mathcal{T}_{i,n}$ preserves the isotopy class and the length
of every curve $\alpha\in\hat{\tau}$. Thus applying
$(\mathcal{T}_{k,n}\circ ...\circ\mathcal{T}_{1,n})^{-1}$ to $\ell_{\alpha}(\varphi_{k,n}(\zeta_{n}(T)))$, we obtain
$$\ell_{\alpha}(\varphi_{k,n}(\zeta_{n}(T)))=\ell_{\alpha}(\psi_{n}(\zeta_{n}(T))).$$
 Then by the previous paragraph for all $n$ sufficiently large,
$\ell_{\alpha}(\psi_{n}(\zeta_{n}(T)))\leq L_{S}$. Thus there is a
Bers pants decomposition  $Q_{k+1,n}$ of $\psi_{n}(\zeta_{n}(T))$
containing $\hat{\tau}$.  

 Let the threshold constant in the distance formula (\ref{eq : dsf}) be $A> \max\{M_{1},R,2\}$. Then there are constants $K\geq 1$ and $C\geq 0$ such that 
 \begin{equation}\label{eq : dsfQ0Qk+1}d(Q_{0,n},Q_{k+1,n})\asymp_{K,C} \sum_{\substack{Y\subseteq S\\ \nonannular}}\{d_{Y}(Q_{0,n},Q_{k+1,n})\}_{A}.\end{equation}
As we saw above $\hat{\tau}\subset Q_{0,n}$ and $\hat{\tau}\subset
  Q_{k+1,n}$. So for any essential subsurface $W$ satisfying $\hat{\tau}\pitchfork W$,
  it follows that
  $$d_{W}(Q_{0,n},Q_{k+1,n})\leq 2.$$
   Thus subsurfaces which overlap $\hat{\tau}$ do not contribute to the right hand side of (\ref{eq : dsfQ0Qk+1}). On the other hand, by Theorem \ref{thm : brockqisom} (Quasi-Isometric Model) there are constants $K_{\WP}\geq 1,C_{\WP}\geq 0$ such that 
   $$d(Q_{0,n},Q_{k+1,n})\asymp_{K_{\WP},C_{\WP}} d_{\WP}(\zeta_{n}(0),\zeta_{n}(T)).$$ 
   Let $T_{0}=  K_{\WP}(KA+KC)+K_{\WP}C_{\WP}$. Since $T\geq T_{0}$ by the above quasi-equality 
   $$d(Q_{0,n},Q_{k+1,n})\geq KA+KC.$$
 Now (\ref{eq : dsfQ0Qk+1}) and the above inequality imply that for any $n\in\mathbb{N}$, there is an essential non-annular subsurface $Y_n$ with 
 $$d_{Y_{n}}(Q_{0,n},Q_{k+1,n})\geq A\geq R.$$
 But as we saw above $Y_n$ can not overlap $\hat{\tau}$ (otherwise $d_{Y_n}(Q_{0,n},Q_{k+1,n})\leq 2<A$), therefore $Y_{n}\subseteq S\backslash \hat{\tau}$. Moreover, since $\hat{\tau}\neq \emptyset$, $Y_{n}$ is a proper subsurface. Applying $\psi_{n}^{-1}$ to the subsurface coefficient above we get
  \begin{equation}\label{eq : dpsinY}d_{\psi_{n}^{-1}(Y_{n})}(\psi_{n}^{-1}(Q_{0,n}),\psi_{n}^{-1}(Q_{k+1,n}))\geq R,\end{equation}
where $\psi_{n}^{-1}(Q_{0,n})$ is a Bers pants decomposition of
$\zeta_{n}(0)$ and  $\psi_{n}^{-1}(Q_{k+1,n})$ is a Bers pants
decomposition of $\zeta_{n}(T)$. Moreover, $\psi_{n}^{-1}(Y_{n})$ is a
proper subsurface of $S$, because $Y_{n}$ is a proper subsurface of $S$. But then the lower bound (\ref{eq : dpsinY})
contradicts the non-annular bounded combinatorics assumption for the
two pants decompositions $Q(\zeta_{n}(0))$ and $Q(\zeta_{n}(T))$. This contradiction completes the proof of the fact that $\hat{\tau}=\emptyset$.

Let $t^{*}=\frac{t_1}{2}$. By Theorem \ref{thm : geodlimit} (\ref{gl : sigma}) and since $\hat{\tau}=\emptyset$, we have that $\hat{\zeta}(t^{*})\in \Teich(S)$.
So $\inj(\hat{\zeta}(t^{*}))> 2\epsilon$ for some $\epsilon>0$. Furthermore, by
Theorem \ref{thm : geodlimit} (\ref{gl : converge}), there is a sequence $\{m_{n}\}_{n=1}^{\infty}$ such that
$\psi_{m_{n}}(\zeta_{m_{n}}(t^{*}))\to \hat{\zeta}(t^{*})$ as $n\to \infty$.
Therefore, $\inj(\psi_{m_{n}}(\zeta_{m_{n}}(t^{*})))> \epsilon$ for any $n$ sufficiently large. Then since the action by
elements of the mapping class group does not change the injectivity
radius of a surface $\inj(\zeta_{m_{n}}(t^{*}))>\epsilon$. The lemma is proved.
\end{proof}

\begin{proof}[Proof of Theorem \ref{thm : recurrence}]Let $T_{0}>0$ be as in Lemma \ref{lem : zetamnthick} and $T\geq T_{0}$. Consider the sequence of WP geodesic segments 
$$\zeta_{n}:=r|_{[nT,(n+1)T]}:[0,T]\to \Teich(S).$$
Note that Theorem \ref{thm : bddcombstable} guarantees that for $D=d_{R}(K_{\WP},C_{\WP})$, the paths $Q(r)$ and $\rho$, $D-$fellow travel in the pants graph. Let $z^{-}_{n},z^{+}_{n}\in[0,\infty)$ be so that 
\begin{eqnarray*}
d(\rho(z^{-}_{n}),Q(\zeta_{n}(0)))&\leq& D,\;\text{and}\\
 d(\rho(z^{+}_{n}),Q(\zeta_{n}(T)))&\leq& D.
 \end{eqnarray*}
Then for every essential non-annular  subsurface $Y\subsetneq S$, 
\begin{eqnarray} d_{Y}(\rho(z^{-}_{n}),Q(\zeta_{n}(0)))&\leq& D,\;\text{and}\label{eq : dYrz-q0} \\
d_{Y}(\rho(z^{+}_{n}),Q(\zeta_{n}(T)))&\leq& D.\label{eq : dYrz+qT}
\end{eqnarray}
Moreover by the assumption that the pair $(\nu^{-},\nu^{+})$ has non-annular $R-$bounded combinatorics for any proper, essential non-annular subsurfaces $Y\subsetneq S$ we have, $d_{Y}(\nu^{-},\nu^{+})\leq R$. Then by the no back tracking property of Hierarchy paths (Theorem \ref{thm : nbacktr}) there is an $M_{2}>0$ so that
\begin{equation}\label{eq : dYrz-rz+}d_{Y}(\rho(z^{-}_{n}),\rho(z^{+}_{n}))\leq R+2M_{2}.\end{equation}
The subsurface coefficient bounds (\ref{eq : dYrz-q0}), (\ref{eq : dYrz+qT}) and (\ref{eq : dYrz-rz+}) combined with the triangle inequality imply that
\begin{eqnarray*}
d_{Y}(Q(\zeta_{n}(0)),Q(\zeta_{n}(T)))&\leq& 2D+R+2M_{2}+\diam_{Y}(\rho(z^{-}_{n}))+\diam_{Y}(\rho(z^{+}_{n}))\\
&\leq& 2D+R+2M_{2}+4.
\end{eqnarray*}
 Thus the pair $(Q(\zeta_{n}(0)),Q(\zeta_{n}(T)))$ has $R+2D+2M_{2}+4$ non-annular bounded combinatorics.

 Then Lemma \ref{lem : zetamnthick} applies to the sequence of geodesic segments $\zeta_{n}:=r|_{[nT,(n+1)T]}$ and implies that there are $t^{*}\in [0,T]$, $\epsilon>0$ and a sequence of integers $\{m_{n}\}_{n=1}^{\infty}$ such that at $a_{n}=m_{n}T+t^{*}$ we have
$$\inj(r(a_{n}))>\epsilon.$$
 This implies that $\hat{r}(a_{n})$ where $\hat{r}$ is the projection of $r$ to the moduli space is in the $\epsilon-$thick part of the moduli space. Furthermore, since $a_{n}\to \infty$, the ray is recurrent to the $\epsilon-$thick part of the moduli space.
\end{proof}

 Let $r:[0,\infty)\to \Teich (S)$ be the ray with end invariant $(\nu^{-},\nu^{+})$ with non-annular bounded combinatorics. In Theorem \ref{thm : recurrence} we saw that the ray $\hat{r}$ is recurrent to a compact subset of $\mathcal{M}(S)$. In Theorem \ref{thm : excursion} we show that if in addition there is a sequence of curves $\{\gamma_{i}\}_{i=1}^{\infty}$ so that $d_{\gamma_{i}}(\nu^{-},\nu^{+})\to\infty$ as $i\to\infty$, then the recurrent ray $\hat{r}$ is not contained in any compact part of the moduli space. The theorem also follows from Theorem 3.1 of \cite{bmm2}. The proof here is different and more direct and gives some information about the excursion times. We need the following result from $\S 4$ of \cite{wpbehavior}.
 \begin{lem} \label{lem : twsh} \textnormal{(Large twist $\Longrightarrow$ Short curve)}
Given $T, \epsilon_{0}$ and $N$ positive, there is an $\epsilon<\epsilon_{0}$ with the following property. Let  $\zeta: [0, T'] \to \Teich(S)$ be a WP geodesic segment of length $T'\leq T$ such that
$$\sup_{t \in [0,T']}\ell_{\gamma}(\zeta(t))\geq \epsilon_{0}$$ 
If $d_{\gamma}(\mu(\zeta(0)), \mu(\zeta(T'))) >N$ ($\mu(X)$ denotes a Bers marking of the point $X\in \Teich(S)$), then we have
  $$\inf_{t \in [0,T']} \ell_{\gamma}(\zeta(t))\leq \epsilon .$$
  Moreover, $\epsilon\to 0$ as $N\to \infty$.
  \end{lem}

\begin{thm}\label{thm : excursion}
Let $r:[0,\infty)\to\Teich(S)$ be a WP geodesic ray with end invariant $(\nu^{-},\nu^{+})$. Suppose $(\nu^{-},\nu^{+})$ has non-annular $R-$bounded combinatorics. Moreover assume that there is a sequence of curves $\{\gamma_{i}\}_{i=1}^{\infty}$ so that 
$$d_{\gamma_{i}}(\nu^{-},\nu^{+})\to\infty$$
 as $i\to\infty$. Then there is a sequence of times $b_{i}\to \infty$ as $i\to \infty$ such that 
$$\ell_{\gamma_{i}}(r(b_{i}))\to 0$$ 
as $i\to \infty$.
\end{thm}
\begin{proof}
Let $\rho:[0,\infty)\to P(S)$ be a hierarchy path with end points $\nu^{-}$ and $\nu^{+}$.  The pair $(\nu^{-},\nu^{+})$ has non-annular bounded combinatorics, so Theorem \ref{thm : bddcombstable} implies that for $D=d_{R}(K_{\WP},C_{\WP})$, $\rho$ and $Q(r)$, $D-$fellow travel. Moreover both $\rho$ and $Q(r)$ are quasi-geodesics. Thus there is a quasi-isometry $N:[0,\infty]\to [0,\infty)$ from the domain of $\rho$ to the domain of $r$. The map assigns to each $i$ in the domain of $\rho$ any time $t$ in the smallest interval in the domain of $r$ which contains all $t'$ with $d(\rho(i),Q(r(t')))\leq D$. For more detail see $\S 5$ of \cite{wpbehavior}. Denote the constants of the quasi-isometry $N$ by $K_{1}$ and $C_{1}$.

We assumed that $d_{\gamma_{i}}(\nu^{-},\nu^{+})\to\infty$ as $i\to\infty$ so for all $i$ sufficiently large
$$d_{\gamma_{i}}(\nu^{-},\nu^{+})\geq M_{1},$$
 where $M_{1}$ is the constant from the Large Link Lemma (\cite[Lemma 6.2]{mm2}), see also property (2) of hierarchy paths in \cite[Theorem 2.6]{bmm2}. Then the annular subsurface with core curve $\gamma_{i}$ is a component domain of $\rho$. Thus there is a time $q_{i}\in [0,\infty]$, so that $\rho(q_{i})$ contains the curve $\gamma_{i}$. Note that the sequence of times $q_{i}\to\infty$ as $i\to\infty$. 
 
 Since the pair $(\nu^{-},\nu^{+})$ has non-annular $R-$bounded
combinatorics, for any proper, essential non-annular subsurfaces $Y\subsetneq S$ we have $d_{Y}(\nu^{-},\nu^{+})\leq R$. Then by the no back tracking property of hierarchy paths (Theorem \ref{thm : nbacktr}) there is an $M_{2}>0$ so that for any $i,j\in\mathbb{N}$ we have
$$d_{Y}(\rho(q_{i}),\rho(q_{j}))\leq R+2M_{2}.$$
 Let the threshold in the distance formula (\ref{eq : dsf}) be $\max\{M_{1},R+2M_{2}\}$. Then there
are constants $K_{R}\geq 1$ and $C_{R}\geq 0$ corresponding to the
threshold so that
\begin{equation}\label{eq : dqiqj}d(\rho(q_{i}),\rho(q_{j}))\asymp_{K_{R},C_{R}} d_{S}(\rho(q_{i}),\rho(q_{j})).\end{equation}
 
Let $\textbf{w}(D,0,R)$ be the constant from the Annular Coefficient Comparison Lemma in $\S 6$ of \cite{wpbehavior}, see below. Let $w=\max\{{\bf w},2\}$. We have that the pair $(\nu^{-},\nu^{+})$ has non-annular $R-$bounded combinatorics. Moreover for each $i$ sufficiently large $\gamma_{i}\in\rho(q_{i})$. These two facts imply that for all $i$ sufficiently large, the curve $\gamma_{i}$ is $(w,0)-$isolated at $q_{i}$, where the subsurface with non-annular $R-$bounded combinatorics on both sides of $q_{i}$ is the surface $S$. See \cite[\S 6.1]{wpbehavior} for the definition of isolated curve (isolated annular subsurface) along a hierarchy path.  
 
Recall that $\rho$ is a $(k,c)-$quasi-geodesic in $P(S)$. Let $K_{2}=\max\{K_{1},K_{R}\}$ and $C_{2}=\max\{C_{1},C_{R}\}$, where $K_{1},C_{1}$ are the constants of the quasi-isometry $N$ and $K_{R},C_{R}$ are the constants in the quasi-equality (\ref{eq : dqiqj}). For any integer $i\geq 0$, set 
\begin{center}
$q^{-}_{i}=q_{i}-k(K_{2}(w+C_{2}))-kc$ and $q^{+}_{i}=q_{i}+k(K_{2}(w+C_{2}))+kc$. 
\end{center}
By the setup of $q_{i}^{-}$ for any $q\leq q_{i}^{-}$ we have $d(\rho(q_{i}),\rho(q))\geq K_{2}(w+C_{2})$, and by the setup of $q_{i}^{+}$ for any $q\geq q_{i}^{+}$, $d(\rho(q_{i}),\rho(q))\geq K_{2}(w+C_{2})$.
Then the quasi-equality (\ref{eq : dqiqj}) implies that
 $$d_{S}(\rho(q_{i}),\rho(q))\geq w\geq 2.$$
The above inequality guarantees that the $\mathcal{C}(S)-$distance of any curve in in the pants decomposition $\rho(q_{i})$ and any curve in in the pants decomposition $\rho(q)$ is at least $2$. Thus any curve $\rho(q_{i})$ intersects any curve in $\rho(q)$. In particular, $\gamma_{i}$ intersects any curve in $\rho(q)$. Then there is an $M_{3}> 0$, so that $d_{\gamma_{i}}(\rho(q^{+}_{i}),\nu^{+})\leq M_{3}$ and $d_{\gamma_{i}}(\rho(q^{-}_{i}),\nu^{-})\leq M_{3}$, see property (4) of hierarchy paths in \cite[Theorem 2.6]{bmm2}. Therefore,
\begin{equation}\label{eq : dgirqi=dginu}d_{\gamma_{i}}(\rho(q_{i}^{-}),\rho(q^{+}_{i}))\asymp_{1,2M_{3}} d_{\gamma_{i}}(\nu^{-},\nu^{+}).\end{equation}
  Let $s_{i}^{-}\in N(\rho(q_{i}^{-}))$ and $s_{i}^{+}\in N(\rho(q_{i}^{+}))$. Since $q_{i}^{+}-q_{i}\geq {\bf w}$ and $q_{i}-q_{i}^{-}\geq {\bf w}$ by Annular Coefficient Comparison Lemma in  \cite[\S6]{wpbehavior} we have 
 \begin{itemize}
 \item $\min\{\ell_{\gamma_{i}}(r(s_{i}^{-})),\ell_{\gamma_{i}}(r(s_{i}^{+}))\}\geq \omega(L_{S})$, where $\omega(a)$ is the width of the collar of a simple closed geodesic with length $a$ on a complete hyperbolic surface provided by the Collar Lemma (see \cite[$\S 4.1$]{buser}), and 
 \item $d_{\gamma_{i}}(Q(r(s_{i}^{-})),Q(r(s_{i}^{+})))\asymp_{1,B} d_{\gamma_{i}}(\rho(q_{i}^{-}),\rho(q_{i}^{+}))$, for a constant $B$ depending only on $D$. 
  \end{itemize}
By the setup of $q_{i}^{-}$ and $q_{i}^{+}$ we have $q^{+}_{i}-q_{i}^{-}\leq L$, where 
$$L=2k(K_{2}(w+C_{2}))+2kc.$$
 Then since $N$ is a $(K_{1},C_{1})-$quasi-isometry the length of the interval $[s_{i}^{-},s_{i}^{+}]$ is bounded above by $K_{1}L+C_{1}$. This fact and the first bullet above allow us
to apply Lemma \ref{lem : twsh} to the geodesic segment
$r|_{[s_{i}^{-},s_{i}^{+}]}$ and conclude that there exists $\epsilon_{i}>0$ depending on the upper bound for the length
of the interval $[s_{i}^{-},s_{i}^{+}]$, the lower bound
$\omega(L_{S})$ in the first bullet above and the value of the annular coefficient
$d_{\gamma_{i}}(Q(r(s_{i}^{-})),Q(r(s_{i}^{+})))$ so that  
$$\inf_{t\in[s_{i}^{-},s_{i}^{+}]}\ell_{\gamma_{i}}(r(t))\leq
\epsilon_{i}.$$ 
Moreover, the second bullet above, the quasi-equality (\ref{eq : dgirqi=dginu}) and the assumption that 
 $$d_{\gamma_{i}}(\nu^{-},\nu^{+})\to\infty$$
  as $i\to\infty$, together imply that 
$$d_{\gamma_{i}}(Q(r(s_{i}^{-})),Q(r(s_{i}^{+})))\to\infty.$$ 
as $i\to\infty$. Then the last statement of Lemma \ref{lem : twsh} guarantees that
$\epsilon_{i}\to 0$ as $i\to \infty$. 

Let $b_{i}\in[s_{i}^{-},s_{i}^{+}]$ be the time that the above infimum is
realized. Then $\ell_{\gamma_{i}}(r(b_{i}))\to 0$ as $i\to \infty$. Moreover since $q_{i}\to\infty$ as $i\to\infty$, we have $q^{-}_{i}\to\infty$ as $i\to\infty$. Thus $s^{-}_{i}\to\infty$ as $i\to\infty$. Then $b_{i}\to\infty$ as $i\to\infty$. This completes the proof of the lemma.   
\end{proof} 

\begin{proof}[Proof of Theorem \ref{thm : recurnue}] 

Let $\{\gamma_{i}\}_{i=0}^{\infty}$ be a sequence of curves as in $\S$\ref{sec : nue} and let $\nu^{+}$ be the minimal filling non-uniquely ergodic lamination in $\mathcal{EL}(S)$ which is determined by the sequence. Let $\nu^{-}$ be a marking containing $\gamma_{0},...,\gamma_{3}$ as in $\S$\ref{sec : nue}. Then  by Theorem \ref{thm : nueg0} (\ref{nueg0 : nabd}), the pair $(\nu^{-},\nu^{+})$ has non-annular $R-$bounded combinatorics. Let $X\in \Teich(S)$ be a point with a Bers marking $\nu^{-}$. By Lemma \ref{lem : rayinfty} there is a geodesic ray $r:[0,\infty)\to \Teich(S)$ with $r(0)=X$ and the forward ending lamination $\nu^{+}$. Then Theorem \ref{thm : recurrence} implies that $\hat{r}$ is recurrent to a compact subset of $\mathcal{M}(S)$. Furthermore, by Theorem \ref{thm : nueg0} (\ref{nueg0 : dgimn}), 
$$d_{\gamma_{i}}(\nu^{-},\nu^{+})\geq \frac{1}{K}e_{i-1}-C.$$
Then since $e_{i}\to\infty$ as $i\to\infty$, we have $d_{\gamma_{i}}(\nu^{-},\nu^{+})\to \infty$ as $i\to\infty$. Thus by Theorem \ref{thm : excursion} the ray $\hat{r}$ is not contained in any compact subset of $\mathcal{M}(S)$.
\end{proof}

 \begin{remark}
Masur's criterion (Theorem \ref{thm : masurcrit}) guarantees that any Teichm\"{u}ller geodesic ray with vertical lamination $\nu^{+}$ is divergent in $\mathcal{M}(S)$. 
\end{remark}

\bibliographystyle{amsalpha}
\bibliography{reference}
\end{document}